\documentclass[a4paper,12pt]{article}
\usepackage{amsmath,amssymb,amsfonts,amscd,amsthm}
\usepackage{graphicx}
\usepackage{psfrag}
\usepackage{dsfont}
\usepackage{blindtext}

\newtheorem{Main}{\color{bleu1} Theorem}
\newtheorem{Mainc}[Main]{\color{bleu1} Corollary}

\usepackage[titletoc, toc, title, page]{appendix}
\usepackage{color}
\usepackage{enumerate}
\usepackage{fancyhdr}
\usepackage{lastpage}
\usepackage[utf8]{inputenc}
\usepackage[top=2cm, bottom=2cm, left=2cm, right=2cm]{geometry}
\usepackage{subcaption}
\usepackage{bbm}
\usepackage{pgf,tikz}
\usepackage{lipsum}
\usepackage{authblk}
\usetikzlibrary{arrows}
\usepackage{color}
\def\red{\color{red}}

\renewcommand\appendix{\par
  \setcounter{section}{0}
  \setcounter{subsection}{0}
  \setcounter{figure}{0}
  \setcounter{table}{0}
  \renewcommand\thesection{Appendix \Alph{section}.}
  \renewcommand\thefigure{\Alph{section}\arabic{figure}}
  \renewcommand\thetable{\Alph{section}\arabic{table}}
}

\usepackage{xcolor}

\definecolor{bluegray}{rgb}{0.4, 0.6, 0.8}

\def\red{\color{red}}

\usepackage{xcolor}

\newtheorem{theorem}{\color{bleu1} Theorem}[section]                             
\newtheorem{lemma}{\color{bleu1} Lemma}[section]                             
\newtheorem{proposition}{\color{bleu1} Proposition}[section]                             
\newtheorem{corollary}{\color{bleu1} Corollary}[section]                             
                             
\newtheorem{remark}{\color{bleu1} Remark}[section]                             
\newtheorem{definition}{\color{bleu1} Definition}[section]                             
\newtheorem{question}{\color{bleu1} Question}[section]                             
\newtheorem{conjecture}{\color{bleu1} Conjecture}[section]                             

\renewenvironment{proof}{\noindent {\bf \color{bleu1} Proof.}}{\qed\vskip 0.2cm}

\newcommand{\C}{\mathbb C}
\newcommand{\R}{\mathbb R}
\newcommand{\Q}{\mathbb Q}
\newcommand{\Z}{\mathbb Z}
\newcommand{\N}{\mathbb N}

\newcommand{\T}{\mathbb{T}}

\newcommand{\cT}{\mathcal{T}}
\renewcommand{\a}{\alpha}

\renewcommand{\min}{\mathop{\mbox{min}}}

\def\mi\mathfrak{i}

\def\DS{\displaystyle}

\def\PL{{\Prob_{\Leb}}}

\def\al{{\alpha}}
\def\eps{{\varepsilon}}
\def\NN{{\mathbb{N}}}
\def\R{{\mathbb{R}}}
\def\QQ{{\mathbb{Q}}}
\def\Q{{\mathbb{Z}}}
\def\T{{\mathbb{T}}}

\def\a{{\alpha}}

\def\N{{\mathbb{N}}}

\def\R{{\mathbb{R}}}
\def\Q{{\mathbb{Q}}}
\def\Z{{\mathbb{Z}}}
\def\T{{\mathbb{T}}}
\def\bbu{\bar u}

\DeclareMathAlphabet{\mathcalligra}{T1}{calligra}{m}{n}

\def\sp{ \mathcalligra p }
\def\sq{ \mathcalligra q }

\def\BS{{\bar {\Sigma}}}



\newcommand{\cO}{\mathcal O} 

\numberwithin{equation}{section}

\title{\color{bleu1} \large{Instabilities of analytic quasi-periodic tori}}
\author{\small{G. Farr\'e, B. Fayad}}
\date{}

\usepackage{sectsty,ulem} \definecolor{bleu1}{RGB}{0,57,128}
\subsectionfont{\color{bleu1}}
\sectionfont{\color{bleu1}}
\subsectionfont{\color{bleu1}}

\usepackage[utf8]{inputenc}
\usepackage[english]{babel}

\usepackage[colorlinks,linkcolor=bleu1,anchorcolor=green,citecolor=bleu1]{hyperref}

\newcommand{\barZ}{\bar Z}

\newcommand{\ignore}[1]{}

\def\C{\mathcal C}

\definecolor{dgreen}{rgb}{0.1,0.6,0.1}

\definecolor{bluegreen}{rgb}{0.1,0.5,0.2}

\def\black{\color{black}}

\def \qed{\hfill$\square$}

\def\eps{{\varepsilon}}

\def\Leb{{\rm Leb}}

\def\Prob{{\mathbb{P}}}

\def\Var{{\rm Var}}

\def\EXP{{\mathbb{E}}}

\def\naturals{\mathbb{N}}

\def\Tor{\mathbb{T}}

\def\reals{\mathbb{R}}

\def\integers{\mathbb{Z}}

\def\bv{\mathbf{v}}

\def\bru{{\bar{u}}}

\def\cA{\mathcal{A}}

\def\cB{\mathcal{B}}

\def\cD{\mathcal{D}}

\def\cJ{\mathcal{J}}

\def\cE{\mathcal{E}}

\def\cO{\mathcal{O}}

\def\cP{\mathcal{P}}

\def\cR{\mathcal{R}}

\def\cT{\mathcal{T}}

\def\fp{\mathfrak{p}}

\def\fq{\mathfrak{q}}

\makeatother

\def\beq{\begin{equation}}
\def\eeq{\end{equation}}

\author{Dmitry Dolgopyat, Bassam Fayad, Maria Saprykina}

\title{\color{bleu1} Erratic behavior for 1-dimensional random walks in a Liouville quasi-periodic environment}

\begin{document}

\maketitle
\color{bleu1}
\begin{abstract}
\black
{We  show that one-dimensional  random walks in a quasi-periodic
environment with Liouville  frequency generically have an erratic statistical behavior. 
In the recurrent case we show that neither quenched nor annealed limit theorems hold
and both drift and variance exhibit wild oscillations, being logarithmic at some times and
almost linear at other times. In the transient case we show that the annealed Central Limit
Theorem fails generically. These results are in stark contrast with the Diophantine case where
the Central Limit Theorem with linear drift and variance was established by Sinai.}
\end{abstract}
\black

\maketitle
\section{Introduction}
\subsection{Quasiperiodic random walks.}
Let   
$C^\infty(\Tor,(0,1))$ be the set of smooth functions from the
standard torus $\T=\R/\Z$ to $(0,1)$ (these can be identified with
smooth $1$-periodic functions from $\R$ to $(0,1)$). 
Each triple $(\fp,\alpha, x)$, where  $\fp \in C^\infty(\Tor,(0,1))$,
$\a \in (0,1)$ and $x\in \T$ defines a sequence of numbers
$\sp(j)=\fp(x+j\a)$ for all $j\in \Z$. The sequence $(\sp(j))_{j\in
  \Z}$ will be called the {\it quasi-periodic environment} defined by
$(\fp,\alpha, x)$, or just {\it environment} $(\fp,\alpha, x)$.

Consider the nearest neighbor random walk $(Z_t)_{t\in \N}$ on the
one dimensional lattice $\Z$, given by $Z_0=0$, and 
 for $t\in \N$
\begin{equation} 
\label{DefMCZ}
\Prob_x(Z_{t+1}=k+1|Z_t=k)=\fp(x+k\a), \quad 
\Prob_x(Z_{t+1}=k-1|Z_t=k)=\fq(x+k\a),
\end{equation}
where $\fq(x)=1-\fp(x)$, and where we stressed the dependence on $x$ with the
notation $\Prob_x$.

Following \cite{S2}, one can also define a related
Markov process $(X_t)_{t\in \N}$ on $\T$: 
\begin{equation}
\label{DefMC_X}
X_t=X_0+Z_t \alpha \ {\rm mod} \ 1, \quad X_0=x.
\end{equation}

When $\a \notin \Q$, we call $(Z_t)_{t\in \N}$ a {\it one-dimensional
  random walk in quasi-periodic environment}, or for short {\it a
  quasi-periodic walk}.
The behavior of these walks has many similarities with that of the classical random
walks in a random environment, and yet many differences. 
We will present the statements in
different contexts, in particular, quenched and annealed limit laws, as
defined below.


\begin{definition} \label{def.prob} 
We use the notation $\Prob_x$ for the distribution of the paths of $(Z_t)_{t\in \N}$ defined in \eqref{DefMCZ}, when $x \in \T$ is fixed, and 
 $\PL$ for the  distribution of the paths of $(Z_t)_{t\in \N}$ when
 $x$ is uniformly distributed on $\Tor$ with respect to the  Haar
 measure ${\rm Leb}$ on $\T$. 
The notation $\mathbb E_x$ is reserved for the expectation under the
probability $\Prob_x$, while $\mathbb E_{\rm Leb}$ is used for the
expectation  under $\PL$. Similarly, we use the notations ${\rm
  Var}_x$ and ${\rm Var}_{\rm Leb}$ to denote variances for $\Prob_x$ and
$\PL$. The same notations will be used for the distribution of 
the paths of $(X_t)_{t\in \N}$ defined in \eqref{DefMC_X}.
\end{definition}

\begin{definition}[Quenched and Annealed limit
  theorems] \label{def.limits} 

Consider a quasi-periodic environment defined by a triple
$(\fp,\a,x)$ where $\fp\in C^\infty(\T,(0,1))$, $\a \notin \Q$ and
$x\in \T$, and let
$(Z_t)_{t\in \N}$ be the random walk defined as in \eqref{DefMCZ}.
When $x \in \T$ is fixed, we say that the walk $(Z_t)_{t\in \N}$ satisfies a quenched limit theorem
  if there  exist  sequences $(b_t(x))_{t\in \N}$ and $(\sigma_t(x))_{t\in \N}$ 
 and a proper
 distribution $\cD_x(\cdot)$ such that for any $z \in \R$
\begin{equation*} 
\lim_{t\to\infty} \Prob_{x} (Z_t-b_t(x) <  \sigma_t(x) z)=\cD_x(z). 
\end{equation*} 

We say that the walk satisfies an annealed limit theorem if there 
exist sequences  $(b_t)_{t\in \N}$ and $(\sigma_t)_{t\in \N}$  and a proper
distribution $\cD(\cdot)$ such that for any $z \in \R$
\begin{equation*} 
\lim_{t\to\infty} \Prob_{\rm Leb} (Z_t-b_t <  \sigma_t z)=\cD(z). 
\end{equation*} 
In analogy with the drift of simple random walks, we call $b_t(x)$ or $b_t$  the {\it drift} at time $t$. 
\end{definition} 

We will see that the notion of symmetry is very important in distinguishing different types of
quasi-periodic walks. Following \cite{S2}, we  adopt the following definition. 

\begin{definition}[Symmetric and asymmetric walks]
Given an environment defined by $(\fp,\a,x)$,  
we call $(Z_t)_{t\in \N}$ defined as in  \eqref{DefMCZ}  a symmetric quasi-periodic walk if $\fp$ satisfies 
\begin{equation}
\label{eq_sym}
\int_\Tor \ln \fp(x) dx=\int_\Tor \ln \fq(x) dx. 
\end{equation}
Otherwise we say that the walk is asymmetric. 
We denote by $\cP\subset C^\infty(\Tor,(0,1))$ the set of functions
satisfying the symmetry condition \eqref{eq_sym}, and by
$\cP^c=C^\infty(\Tor,(0,1))\setminus \cP$---the
set of asymmetric walks. \label{defP}
\end{definition} 

Note that in general, even if $\fp \in \cP$, the sequence
$(\fp(x+k\a))_{k\in \Z}$ does not exhibit any symmetry. The effect of
the symmetry condition above is the result of the averaging effect,
which comes from the equidistribution of any orbit of an irrational
rotation on the circle.
It is a classical fact in dynamical systems that the effectiveness of 
the equidistribution of the orbits of an irrational rotation of the
circle is determined by the arithmetic properties of its angle $\a$.

Recall that $\a \in \R$ is said to be {\it Diophantine} (denoted $\a \in {\rm DC}(\gamma,\tau)$) if there exists $\gamma>0$ and $\tau\geq 0$ such that for any $(p,q)\in \Z \times \N^*$
\begin{equation} \label{eq.DC} \tag{\rm{DC}} \left|\a-\frac{p}{q}\right|\geq \frac{\gamma}{q^{2+\tau}}. \end{equation}

An irrational real number that is not Diophantine is called {\it Liouville}.

An elementary but noticeable fact of number theory is that Liouville numbers form a dense
$G^\delta$ set of $\R$, while the Lebesgue measure of this set is zero. 

We now introduce a short hand notation that will help simplify the exposition. 

\begin{definition}[Diophantine and Liouville walks]
We call $(Z_t)_{t\in \N}$ defined as in  \eqref{DefMCZ} with $\a \in {\rm DC}$  a Diophantine walk. If  $\a \notin \Q \cup {\rm DC}$, we call the walk Liouville. 
\end{definition}

Because the equidistribution of Diophantine rotations is more effective than that of the Liouville ones, the averaging on the medium due to quasi-periodicity is more effective in the case of  Diophantine quasi-periodic environment. As a result, they behave similarly to simple random walks, as was established by Sinai in \cite{S2}. Namely,
they have linear (possibly null) drift and satisfy the Central Limit Theorem. The precise statements 
for Diophantine walks will be recalled in \S \ref{sec.dioph}.
By contrast, nothing was known for Liouville environments
beyond the results which hold in any uniquely ergodic environment.

In this paper we study the Liouville case.
We show that for any given Liouville $\a$ the walk $(Z_t)_{t\in \N}$ defined by \eqref{DefMCZ}
with a generic $\fp \in C^\infty(\Tor,(0,1))$  has a very erratic statistical behavior. 
By {\it generic} we mean of first category for the $C^\infty$ topology. Since a generic irrational $\a$ is Liouville, our results imply the erratic behavior for one-dimensional random walks in a generic quasi-periodic environment. What we mean by erratic, is that at different times scales the walk $(Z_t)_{t\in \N}$ behaves very differently (think of a walk that drifts almost linearly for a subsequence of times $(t_n)_{n\in \N}$, while for another subsequence $(t'_n)_{n\in \N}$ it will be localized logarithmically around the origin).

  Our main results can be summarized as follows. For symmetric random walks we show that the following behavior is generic:

$\color{bleu1} \bullet$ The spread of the walk (as measured, for example, by standard deviation) oscillates wildly. 
Sometimes the walk is localized at a logarithmic scale while at other times the variance grows faster
than $t^{1-\eps}$. At the latter scales the walk bypasses both $-t^{1-\eps}$  and $t^{1-\eps}$ with probability larger than $0.1$.

$\color{bleu1} \bullet$ The drift of the walk oscillates wildly: sometimes it is larger than $t^{1-\eps}$,
sometimes it is smaller than $-t^{1-\eps}$, sometimes it is of order 1.

$\color{bleu1} \bullet$ The walk does not satisfy neither an annealed nor a quenched limit theorem: the set of limit distributions
includes the normal distribution as well as a distribution with atoms. \smallskip

We will also show that

$\color{bleu1} \bullet$ A one-dimensional random walk in a generic asymmetric quasi-periodic environment 
does not have an annealed limit law. \smallskip

The precise statements of the results outlined above are contained in \S \ref{res.liouv}.

\ignore{Let us give a deliberately vague statements here that will be made precise in Theorems A,B,C and D of \S \ref{res.liouv}.

\medskip

\noindent {\it A one-dimensional random walk in a generic symmetric quasi-periodic environment has an erratic statistical behavior : For almost every starting point $x$, there exist time sequences $\{u_n\},\{v_n\},\{w_n\},\{y_n\}$  such that $Z_{u_n}$ is logarithmically confined, $Z_{v_n}$ satisfies a CLT with almost linear positive drift and almost linear variance, $Z_{y_n}$ behaves like a simple symmetric random walk with a centered CLT with almost linear variance, $Z_{w_n}$ has no limit law.} 

}

\medskip

\noindent{\sc \color{bleu1} Plan of the paper and outline of the
  proofs.} Section \ref{sec.statements} contains all the main statements and a review of related resutls from the literature. Quasi-periodic Diophantine 
environments are discussed in \S \ref{sec.dioph}. \S  \ref{res.liouv} 
contains the precise statements about Liouville walks. It turns out
that their behavior is quite different 
from the Diophantine walks, and is, in fact, quite similar to the
walks in a generic  deterministic elliptic environment that we define and discuss in \S \ref{SSGen}. The Liouville walks are more
erratic than the walks in independent random environments that we briefly review in \S \ref{SSIID}.
Several open questions motivated by the present work are discussed in \S \ref{SSOpen}.


The {\bf main technical tool}  in the study of one dimensional random walks in a
fixed environment given by transition probabilities $\sp_j, j \in \Z$ and
$\sq_j=1-\sp_j$, is the martingale
\eqref{EqMart}. The important quantities that are involved in this
martingale, and that determine the behavior of the walk, are the sums 
\begin{equation}
\label{Pot}
\Sigma{(n)}=
\begin{cases}
{\DS \sum_{j=1}^{n} \ln \sq_j-\ln\sp_j }&\text{if } n\geq 1, \\
0, & \text{if }n=0, \\
{\DS \sum_{j=n+1}^{0} \ln \sp_j-\ln\sq_j }& \text{if }n\leq -1.
\end{cases}
\end{equation}
The function $n\mapsto \Sigma(n)$ is known as the {\it potential}. 
A direct inspection shows that 
if $\sp_j>\sq_j$ for all $j$ in some interval $I\subset \Z$ then $\Sigma$ is decreasing on $I$, 
while if
$\sq_j>\sp_j$ for all $j\in I$ then $\Sigma$ is increasing on $I$. 
Thus, the guiding intuition is that the walker tends to go downwards
on the graph of $\Sigma$ and spends a 
lot of time near local minima of the potential.  
The study of the potential plays a crucial role in the study
of random walks on $\Z$ starting with the pioneering work of Sinai
\cite{S1} and it is also central in the present paper.

\medskip 

\noindent {\bf \color{bleu1} The reason for the strong difference of behavior between Diophantine and Liouville walks.}    In the case of quasi-periodic walks defined by some 
$\fp \in C^\infty(\T,(0,1))$ as in \eqref{DefMCZ}, both the transition probabilities
$\sp_{x,j}=\fp(x+j\a)$ and the sums $\Sigma_x(n)$ defined by
\eqref{Pot} are dependent on $x\in \T$.

If the walk is symmetric (see \eqref{eq_sym}) and $\a$ is
Diophantine, the fact that 
$\ln (1-\fp) -\ln \fp$ is a smooth coboundary  above $R_\a$ implies that the
sums $\Sigma_x(n)$ are bounded, which renders the Diophantine walk
very similar
to the simple symmetric random walk 
(we will come back to this in \S \ref{sec.dioph} and we refer the
reader to \cite{DG4, DG6} for more details). 

{\it A contrario,} obtaining various specific behaviors for the sums
$\Sigma_x(n)$ of a generic function $\fp \in  C^\infty(\T,(0,1))$ when $\a$ is
Liouville, 
underlies all our findings. Displaying very different behaviors of
$\Sigma_x(n)$ at different time scales $n$ and different initial
conditions $x$ is the key behind the erratic behavior of the Liouville walks.

\medskip 

\noindent {\bf \color{bleu1}  Outline of the proofs.} Section \ref{ScPrel} contains the necessary preliminaries.  To keep the exposition as clear as possible, we split the analysis 
of Liouville walks into two separate parts. 
In the first part (Sections \ref{ScAbstPot}--\ref{ScCrit})  
we deal with a fixed environment and
describe several criteria based  on the behavior of the potential, 
that imply various types of behaviors for the random walk. 

In Section \ref{ScAbstPot}, we formulate criteria for localization,
one-sided drift and two-sided drift  for random walks in a fixed
environment. The proofs are given in Section \ref{ScCrit}; they
rely on 
auxillary  estimates of exit times for random walks in a fixed 
environment presented in Section \ref{ScExit}. 

The second part of the paper (Sections \ref{ScConstr}--\ref{Sec_ABC}) 
deals with quasi-periodic walks. In Section \ref{ScConstr}
we  prove Theorem \ref{Cond_to_behaviors}, stating that when $\alpha$
is Liouville, then for a residual set of symmetric environments, the criteria for localization, one-sided drift and
two-sided drift are satisfied for almost every $x\in \T$. In fact, as
we mentioned above, the criteria ask for particular  behaviors of the
sums $\Sigma_x(n)$ at different time scales $n$ and different initial
conditions $x$. 
By definition of the criteria, it will be easy to show that the set
of $\fp \in  \cP$ for which these criteria are satisfied contains a countable
intersection of countable unions of open sets. These open sets, are
subsets of  $\fp \in \cP$ for which a criterion on $\Sigma_x(n)$ holds for some
$n$ and some (not too small)  intervals of  initial conditions $x$.

To prove the theorem, we just need to show that the union of these
open sets is dense. For this we start by  perturbing any given $\fp
\in \cP$ into a smooth multiplicative coboundary  $\bar \fp$ above $R_\a$. 
Then, the main construction is to show that  any smooth coboundary  
$\bar \fp$ can be perturbed to $\bar \fp+e_n(\cdot) \in \cP$ that
satisfies each of the above mentioned criteria at different scales. This is stated in the main 
Proposition \ref{PropCond}. \S \ref{SSRedCob} contains the reduction of Theorem \ref{Cond_to_behaviors} 
to Proposition \ref{PropCond} while the rest of 
 Section \ref{ScConstr} is devoted to the proof of 
Proposition \ref{PropCond}. The proof 
proceeds by a Liouville construction in which we obtain $e_n$ and prove 
the required properties of the  ergodic sums as in        
\eqref{Pot}  for the function $\bar \fp+e_n$.

Section \ref{Sec_ABC} contains the proofs of all the statements
about Liouville walks, including the asymmetric walks that are treated
in \S \ref{SSAsym}. We note that the proofs in \S \ref{SSAsym} do not
use the results of Sections \ref{ScAbstPot}--\ref{ScCrit} so the
reader who is only interested in the asymmetric walks could skip those sections.

In the appendix we give the proof of the results in the generic deterministic elliptic medium. They are similar, albeit easier because we have more freedom in perturbing the medium, to the proofs in the Liouville environments.

\section{Results}  \label{sec.statements}
{In this section we present our main results about Liouville walks and
compare them with other classes of random walks. We start with a summary on what is known for Diophantine walks. Then we state our results on Liouville walks. After that we give a brief list of erratic behaviors of a  walk in a generic deterministic elliptic medium. Next, we ask some natural questions that arise from our results. Finally, we end the section with a very brief survey of related results in the case of random walks in independent random media. }

\subsection{Diophantine walks. } \label{sec.dioph}  
In this section we review the known 
results about quasi-periodic Diophantine
walks. These results show that Diophantine walks 
are very similar to the simple random walks.

Recall the notations 
introduced in Definition \ref{def.prob}. In particular, recall the notation $\cP^c=C^\infty(\Tor,(0,1)) \setminus \cP$ where $\cP$ is the set of functions satisfying the symmetry condition
\eqref{eq_sym}. The following results are known.
\begin{theorem} {\it (\cite[Theorem 1]{S2})}
\label{ThAC}
{\it \bf (Stationary measure)} For  $\alpha \notin \Q$ and $\fp \in  C^\infty(\Tor,(0,1))$ such that any of the following two conditions holds:

(1) {\red $\fp \in \cP^c$,}

(2) $\alpha$ is Diophantine,

then there exists a unique probability measure $\nu$ on $\T$ that is stationary  for the process $(X_t)_{t\in \N}$ defined as in \eqref{DefMC_X}, and  this measure is absolutely continuous with respect
to the Haar measure on $\T$.   Moreover, for each $\phi\in C^0(\Tor,\R)$ and for any $x\in \T$ 
$$ \lim_{t\to \infty} \EXP_x(\phi(X_t))= \nu(\phi). $$
\end{theorem}

The precise definition of a stationary measure will be given in \S \ref{SSIM}.

For Diophantine frequencies, the walk $(Z_t)_{t\in \N}$ satisfies the  Central Limit Theorem, 
as shown by the following two statements. Denote
\begin{equation} \label{eq.normal} 
\Phi(z)=\int_{-\infty}^z \frac{1}{\sqrt{2\pi}} e^{-u^2/2}du. \end{equation}

\begin{theorem} {\it (\cite{Al}, \cite[Equations (15) and (20)]{S2})}
\label{ThAnn}
For  $\alpha$ Diophantine  and $\fp \in  C^\infty(\Tor,(0,1))$,   and the walk $(Z_t)_{t\in \N}$  defined as in \eqref{DefMCZ}, there exist  $\bv\in\reals$ and $\sigma>0$ such that
for all $x$
\begin{equation}
\label{EqQuenched}
 \lim_{t\to\infty} \Prob_x(Z_t-t \bv<  \sigma \sqrt{t} z)=
\Phi(z). \end{equation} 
Therefore
\begin{equation}
\label{EqAnn}
 \lim_{t\to\infty} \PL(Z_t-t \bv<  \sigma \sqrt{t} z)=\Phi(z). 
\end{equation}
Moreover $\bv=0$ iff the walk is symmetric.
\end{theorem}

Note that Theorem \ref{ThAnn} shows that the Diophantine walks behave similarly
to simple random walks  independently of the starting point on the circle. Namely, they have linear growth, variance also grows linearly and the limit distribution is Gaussian.

In fact, asymmetric walks $(Z_t)_{t\in \N}$ for {\it all} irrational frequencies,  have quenched limits similar to simple random walks, but with a drift that depends on the starting point. 

\begin{theorem} {\it (\cite[Theorem B.2]{DG4})}  
\label{ThQLT} 
For  $\alpha \notin \Q$ and $\fp\in \cP^{c} $, and the walk  $(Z_t)_{t\in \N}$ defined as in \eqref{DefMCZ}, there exist functions $(b_t(\cdot))_{t\in \Z}$ and
a number $\sigma>0$ such that for any $z\in \R$
\begin{equation}\label{eq.asym} 
\lim_{t\to\infty} \Prob_x(Z_t-b_t(x) <  \sigma \sqrt{t} z)=\Phi(z). 
\end{equation} 
In case $\a$ is Diophantine, we have that for every $x$, we can take $b_t(x)=t\bv$ where 
$\bv$ is given by Theorem \ref{ThAnn}.
\end{theorem}

\begin{remark} A general formula for $b_t(\cdot)$ will be recalled in \S \ref{SSAsym} (see equation \eqref{QDrift}). \end{remark}

The results of Theorems \ref{ThAC}--\ref{ThQLT} have been extended to random walks driven by rotations of $\T^d$, for arbitrary $d\in \N,$
to random walks with bounded jumps where the walker can move from $x$ to $x+j\alpha$ with
$|j|\leq L$ for some $L>1$ and to quasi-periodic walks on the strip, see
\cite{Br2, DG4, DG5, DG6}.

\subsection{Liouville walks.} \label{res.liouv} 
Theorems \ref{ThAC}, \ref{ThAnn} and \ref{ThQLT}  naturally raise the question of what would be the behavior of 
a quasi-periodic walk when the driving frequency $\a$ is Liouville. 
The following statements show that their behavior can indeed be very
different from the Diophantine case.


Recall that  $\cP\subset C^\infty(\Tor,(0,1))$ the set of functions
satisfying the symmetry condition \eqref{eq_sym}, and  $\Phi(z)$ is
the normal 
distribution given by \eqref{eq.normal}.

{\begin{Main}
\label{ThFlMV} 
For any Liouville $\alpha$ there exists  a dense $G_\delta$ set 
$\cR=\cR(\a) \subset \cP$, and for each $\fp\in \cR(\a)$ a set 
$S(\fp)\subset \T$ 
of full measure with the following property.  

Let $\fp\in \cR(\a)$ and $x\in S(\fp)$, and consider a quasi-periodic walk
$(Z_t)_{t\in \N}$ as in \eqref{DefMCZ} in the quasi-periodic environment  defined by $(\fp,\alpha, x)$.
Then 
there exist strictly increasing sequences of integers $(r_n)_{n\in \N}$,
$(s_n)_{n\in \N}$, $(t_n)_{n\in \N}$, and a sequence of positive
integers $(\eps_n)_{n\in \N}$ with $\eps_n \to 0$,  such that for any
$\eps>0$ and for $k$ sufficiently large we have:

(a) {\bf (Localization)} \ For $T=r_n$, it holds

\begin{equation}
\label{EqMLoc}
\Prob_x \left( \max_{t\leq T}|Z_{t}|>16(\ln T)^2 \right) <T^{-2}
      \text{ and } \
      \Var_x(Z_{T})< 300 (\ln T)^4 ; 
\end{equation}
\smallskip 

(b) {\bf (One-sided positive drift)} \ For $T=s_n$ and some $\mu_n(x) > T^{1-\varepsilon}$, for any $z\in \R$, it holds

\begin{equation}
\label{EqOneSideDr}
\left| \Prob_x\left( \frac{Z_{T}-\mu_{n}(x)}{\sigma_n(x)} <  z\right) - \Phi(z) \right|
<\varepsilon, \quad \left| \frac{\ln\sigma_n(x)}{\ln
    T}-\frac{1}{2}\right|<\varepsilon
\end{equation}
where $\sigma_n(x)=\sqrt{\Var_x(Z_{T})}$;
\smallskip 

\ignore{\red (b') {\bf (One-sided negative drift)} \ For $T=s'_n$ and some $\mu'_n(x) < -s_n^{1-\varepsilon}$, for any $z\in \R$, it holds

\begin{equation}
\label{EqOneSideDr}
\left| \Prob_x\left( \frac{Z_{T}-\mu'_{n}(x)}{\sigma'_n(x)} <  z\right) - \Phi(z) \right|
<\varepsilon, \quad \left| \frac{\ln\sigma'_n(x)}{\ln
    T}-\frac{1}{2}\right|<\varepsilon
\end{equation}
where $\sigma'_n(x)=\sqrt{\Var_x(Z_{T})}$.}
\smallskip

{(c) {\bf (Two-sided drift)} For $T=t_n$ there exist $b_n(x),b'_n(x) \in [0.3T^{1/5},0.4 T^{1/5}]$ and $\eps_n \to 0$ such that 

\begin{equation}
\label{EqTwoSideDr}
\Prob_x \left( |Z_T - b_n(x)|<\eps_n T^{1/5}\right)>0.1, \quad   
\Prob_x \left( |Z_T + b'_n(x)|<\eps_n T^{1/5} \right) >0.1.
\end{equation}}

\end{Main}
Naturally,  statement (b) can be modified  to provide a one-sided
{\it negative} drift for the walks over a subsequence of times.  

We will also need a different  version of property (b)  in order to
guarantee the absence of an annealed limit for the walk. Namely, we need a sequence of times so that
property (b) holds with uniform drift parameters for a set of positive measure of initial positions $x\in \T$. 
  {
\begin{Main}
\label{ThFlMV2} For any Liouville $\alpha$, there exists a dense $G_\delta$ set 
$\mathcal R'=\mathcal R'(\a) \subset \cP$ with the following property.

(a) There exists a strictly increasing sequence of integers $(u_n)_{n\in \N}$, and sequences $(\sigma_n)_{n\in \N}$, $(\eps_n)_{n\in \N}$ with $\eps_n \to 0$, such that for any $\fp\in \mathcal R'(\a)$, for any $x\in \T$,  
the  quasi-periodic random walk
$(Z_t)_{t\in \N}$ as in \eqref{DefMCZ} in the environment defined by 
$(\fp,\alpha, x)$ 
satisfies for each $T=u_n$ and any $z\in \R$:
\begin{equation} \label{eq.simpl} 
\left| \Prob_x(Z_T <  \sigma_n \sqrt{T} z) - \Phi(z) \right|
<\eps_n. \end{equation}

\medskip 

(b) There exists a strictly increasing sequence of integers $(v_n)_{n\in \N}$ and two sequences of measurable subsets of $\T$, $({\mathcal I}_{n})_{n\in \N}$, $({\mathcal I'}_{n})_{n\in \N}$ 
with ${\rm Leb}({\mathcal I}_{n})>0.001$ and  ${\rm Leb}({\mathcal I'}_{n})>0.001$,  such that for any $\fp\in \mathcal R'(\a)$, for any $x \in {\mathcal I}_{n}\bigcup {\mathcal I'}_{n}$ 
there exist  $\mu_n(x)$ such that  $\mu_n(x)> v_n^{1-\varepsilon_n}$ if $x\in {\mathcal I}_{n}$ and  $\mu_n(x) < -v_n^{1-\varepsilon_n}$ if $x\in {\mathcal I'}_{n}$, and such that the following holds.

The  quasi-periodic random walk
$(Z_t)_{t\in \N}$ as in \eqref{DefMCZ} in the environment defined by 
$(\fp,\alpha, x)$ 
satisfies for each $T=v_n$ and any $z\in \R$:
\begin{equation}
\label{EqOneSideDr2}
\left| \Prob_x\left( \frac{Z_{T}-\mu_{n}(x)}{\sigma_n(x)} <  z\right) - \Phi(z) \right|
<\varepsilon_n, \quad \left| \frac{\ln\sigma_n(x)}{\ln
    T}-\frac{1}{2}\right|<\varepsilon_n,
\end{equation}
where $\sigma_n(x)=\sqrt{\Var_x(Z_{T})}$.
\medskip

\end{Main}}

{\red Part $(a)$ of Theorem \ref{ThFlMV2} is based on the fact that for a generic $\fp \in \cP$, the quasi-periodic walk $(Z_t)_{t\in \N}$ "simulates" (uniformly in $x\in \T$) a Diophantine walk for a sequence of times. The same phenomenon for asymmetric walks is observed in Theorem \ref{th.asym}. Theorem \ref{Th_B} encloses a similar observation, this time for every $\fp \in \cP$ and for a generic set of $\a\in \R$.}

As a byproduct of our analysis, we will show that the existence of an
absolutely continuous stationary  measure (see \S \ref{SSIM} for the precise definition) is incompatible with the erratic behavior of Theorem~\ref{ThFlMV}.

\begin{Mainc}
\label{CrNoIM}
If $\alpha$ and $\fp$ are as in Theorem \ref{ThFlMV}, then the process $(X_t)_{t\in \N}$ defined by  \eqref{DefMC_X} has no absolutely continuous stationary measure on $\T$. 
\end{Mainc}
The proof of Corollary \ref{CrNoIM} will be given in \S
\ref{SSNoIM}. Its statement makes a sharp 
contrast with the behaviour for Diophantine walks, described in Proposition 
\ref{PrIM}.

Theorems \ref{ThFlMV} and \ref{ThFlMV2}
imply that the walk does not satisfy any limit theorems.
Let us give a more precise statement, with Definition \ref{def.limits} in mind.

{ \begin{Mainc} \label{cor_nolimit}  If $\alpha$ and $\fp$ are as in Theorem \ref{ThFlMV}, then for $x$ in a set of full measure
the walk  $(Z_t)_{t\in \N}$ defined by  \eqref{DefMCZ} in the quasi-periodic environment defined by  
$(\fp,\alpha, x)$
has no quenched limit theorem at $x$.

{\red If $\alpha$ and $\fp$ are as in Theorem \ref{ThFlMV2},
then the walk $(Z_t)_{t\in \N}$ defined as in  \eqref{DefMCZ}  has no annealed limit theorem.}
\end{Mainc}

\begin{proof} 
We start with the absence of quenched limit theorems.  
Consider $x$ for which  (b) and (c) of  Theorem \ref{ThFlMV} hold. 
It follows from (b) that if a quenched limit theorem holds, then the
limit distribution should be normal. 

On the other hand, 
let $t_n$ be as in (c). Then, since $b_n(x)+b'_n(x) \geq
0.6t_n^{\frac{1}{5}}$, the normalization has to be at least of order
$t_n^{\frac{1}{5}}$. Indeed, if the normalization was negligible compared to $t_n^{\frac{1}{5}}$ for infinitely many $n$, then the limit distribution would have to give weight larger than $0.1$ to segments that accumulate at infinity, which is impossible. But if the normalization is comparable to  $t_n^{\frac{1}{5}}$ or larger, then the limit distribution should give a probability 
larger than $0.1$ to two intervals, each one having  
size at most $\eps_n $.
This implies that any limit point along the sequence $t_n$ 
has non-trivial atoms, so it cannot be normal, giving
a contradiction. 

\medskip 
  
To show that the walk has no annealed limit, we use Theorem
\ref{ThFlMV2}. 
From (a) it follows that if an annealed limit theorem holds, 
the limit must be the normal distribution. 

On the other hand, (b) forces the normalization of any candidate annealed limit theorem to be
of order at least $v_n^{1-\eps}$ for the time $v_n$. Indeed, if the normalization was negligible compared to 
$v_n^{1-\eps}$ for infinitely many $n$, then by \eqref{EqOneSideDr2} the limit distribution should give  weight larger than $0.001$ to segments that accumulate at infinity, which is impossible.  But if the normalization is comparable to $v_n^{1-\eps}$ or larger, then the CLT \eqref{EqOneSideDr2} would  imply that the annealed
limit distribution gives a positive mass to an interval of 
size $v_n^{-1/2+\eps}$.  
 This implies that any limit point along the sequence $v_n$ 
has non-trivial atoms, giving
a contradiction with (a).  \color{bleu1}  \color{bleu1}  \end{proof} \black \black}

Our next result shows that for a generic Liouville frequency, 
the behavior of the corresponding walk ``simulates'' that of a Diophantine one (described by Theorem \ref{ThAnn}) for long periods of time. 

\begin{Main}\label{Th_B}
There exists a dense $G_\delta$ set 
$\mathcal A \subset \R$ such that for any $(\alpha, \fp)\in
\mathcal A \times\cP$  there exist sequences $(\sigma_n)_{n\in \N}$ 
and $(T_n)_{n\in \N}$, $T_n \to \infty$, such that for any $x \in \T$
the quasi-periodic random walk 
$(Z_t)_{t\in \N}$ as in  \eqref{DefMCZ} in the environment defined by 
$(\alpha, \fp, x)$ 
satisfies for any $z\in \R$:

\begin{equation} \label{eqdiopha} 
\left| \Prob_x(Z_t <  \sigma_n \sqrt{t} z) - \Phi(z) \right|
<\frac{1}{n} \quad \text{for all } \ t\in[ T_n,e^{T_n}],
 \end{equation}
where $\Phi(z)$ is the normal distribution given by \eqref{eq.normal}.
\end{Main}

{Consider $t$ of order $T_n$. Relation
  \eqref{eqdiopha} implies that  $\sigma_n$ is necessarily larger than
  $1/\sqrt{T_n}$. 
Therefore, when $t$ is of order $e^{T_n}$, we see that the variance is
almost linear in $t$. 
}
\smallskip

Let us turn to the asymmetric quasi-periodic walks. Recall that, by
Theorem~\ref{ThQLT}, a quenched CLT \eqref{eq.asym}  holds 
with
some function $b_t(\cdot)$. Moreover, by Theorem \ref{ThAnn}, in the
Diophantine case there exist
$b_t(\cdot)\equiv t\bv$ such that the annealed limit  \eqref{EqAnn}
holds. To show that no annealed limit theorem holds in the
Liouville case, it suffices to show that the drift function $b_t$  in
the quenched limit theorem  fluctuates
much more than $\sqrt{t}$
when we vary $x$. 

{
 \begin{Main}\label{th.asym}
For any Liouville $\alpha$ there exists a dense $G_\delta$ set 
$\mathcal D (\alpha) \subset \cP^c$, such that the walk $(Z_t)_{t\in \N}$ defined by
\eqref{DefMCZ} 
with $\fp \in \mathcal D (\alpha)$, satisfies the following: 

(a)  There exists a sequence of integers $(s_n)_{n\in \N}$ 
and sequences $(b_n)_{n\in \N}$, $(\sigma_n)_{n\in \N}$,
$(\eps_n)_{n\in \N}$ with $\eps_n \to 0$, 
such that for $T=s_n$ we have for any $x\in \T$ and any $z\in \R$:
\begin{equation} \label{eq.simpl2}
\left| \Prob_x(Z_T -b_n <  \sigma_n \sqrt{t} z) - \Phi(z) \right|
<\eps_n.
\end{equation} 

\medskip

(b)  There exists a 
sequence of integers $(t_n)_{n\in \N}$, and sequences $(\cJ_n)_{n\in
  \N}$ and $(\cJ'_n)_{n\in \N}$ of measurable subsets of $\T$, 
such that the drift coefficients $b_{t_n}(\cdot)$ and $b_{s_n}(\cdot)$ of the quenched CLT \eqref{eq.asym}  satisfy: 
\begin{itemize}
\item[$(i)$] $\mu(\cJ_n)>0.8$ and $\mu(\cJ'_n)>0.1$;
\item[$(ii)$] For $x\in \cJ_n$ we have $|b_{{t_n}}(x)-b_n|<{t_n}^{1/4}$, and
  for $x\in \cJ'_n$ we have $b_{{t_n}}(x)> b_n+{t_n}^{0.9}$.
  \end{itemize}

\ignore{There exist
sequences $(t_n)_{n\in \N}$, $(s_n)_{n\in \N}$, $(b_n)_{n\in \N}$, $(b'_n)_{n\in \N}$, $s_n,t_n \to \infty$, $b_n,b'_n>0$, and sequences $(\cA_n)_{n\in \N}$, $(\cB_n)_{n\in \N}$, $(\cA'_n)_{n\in \N}$, $(\cB'_n)_{n\in \N}$ of measurable subsets of $\T$, such that the drift coefficients $b_{t_n}(\cdot)$ and $b_{s_n}(\cdot)$ of the quenched CLT \eqref{eq.asym}  satisfies 
\begin{itemize}
\item[$(a)$] $\mu(\cA_n)>0.8$ and $\mu(\cB_n)>0.1.$
\item[$(b)$] For $x\in \cA_n$, $|b_{{t_n}}(x)-b_n|<{t_n}^{1/4}$, and
  for $x\in \cB_n$, $b_{{t_n}}(x)> b_n+{t_n}^{0.9}$.
  \item[$(a')$] $\mu(\cA'_n)>0.4$ and $\mu( \cB'_n)>0.4.$
\item[$(b')$] For $x\in \cA'_n$, $|b_{{s_n}}(x)-b'_n|<{s_n}^{1/4}$, and
  for $x\in \cB'_n$, $b_{{s_n}}(x)> b'_n+{s_n}^{0.9}$.
\end{itemize}}
\end{Main}}

As a consequence we get:
\begin{Mainc}\label{cor.asym} For  $\alpha$ and $\fp$ as in Theorem \ref{th.asym}, the walk $(Z_t)_{t\in \N}$ defined by  \eqref{DefMCZ}  has no annealed limit.
\end{Mainc}

{  \begin{proof} If an annealed limit as in  \eqref{EqAnn}
holds, then by (a) the limit must be the normal distribution. On the other hand, properties $(i)$ and $(ii)$ of (b)  require that the normalizing factor at time $t_n$  should satisfy $\sigma_{t_n}
> t_n^{0.9}$. 
But then we get, again from  $(i)$ and $(ii)$, that the limit distribution must give a mass larger than $0.8$
to some point on the line, a contradiction with (a).  
\color{bleu1}  \end{proof} \black}

\subsection{Erratic behavior for random walks in generic deterministic elliptic environments.} \label{SSGen}

In this section,  we consider  random walks on $\Z$ in a fixed
{\it generic} deterministic elliptic environment $(\sp(j))_{j\in \Z}$, that we define as follows.

{ 
\begin{definition}[Deterministic elliptic environments and generic walks]\label{Def_set_E}
For any $\eps \in (0,\frac{1}{2})$, we define $\cE_\eps=\{\sp:\Z\to [\eps, 1-\eps]\}$. We call 
\begin{equation}
\label{set_E}
\cE=\bigcup_{\eps>0} \cE_\eps 
\end{equation}
the set of  {\it deterministic elliptic environments}.
For every $\eps>0$, we endow $\cE_\eps$ with the product topology generated by the sets of the form
$$
W_{\delta, K}(\bar \sp)=
\{\sp: |\sp(n)-\bar \sp(n)|<\delta \text{ for }
|n|\leq K\}.
$$
A subset $\bar{\mathcal E}$ of $\cE$ is called generic if 
for each $\eps$, $\bar{\mathcal E}\cap \cE_\eps$
contains a
countable intersection of open dense sets. 
\smallskip 

For $\sp \in \cE$, we let $(\barZ_t)_{t\in \N}$  be the random walk on $\Z$
 given by $\barZ_0=0$ and  
\begin{equation}
\label{DefMCbarZ}
\Prob(\barZ_{t+1}=k+1|\barZ_t=k)=\sp_k, \quad 
\Prob(\barZ_{t+1}=k-1|\barZ_t=k)=\sq_k. 
\end{equation}
{\red We say that a property is satisfied by a generic random walk on $\Z$ if it is satisfied by the walks \eqref{DefMCbarZ} for $\sp$ in a generic subset of $\cE$. }
\end{definition}

In the definition of generic sets of environments we are allowed to perturb any given environment outside of a finite set, so we
can directly prescribe the asymptotic behavior of the potential to enforce a desirable behavior of the walk.
Hence the erratic behavior similar to the Liouville walks can be observed for random walks in generic deterministic elliptic environments. 
The results are even stronger in the latter case and the proofs are much easier.  
Since we could not find these results in the literature, we give
the proofs in the appendix.

The criteria that yield erratic behavior of generic walks that are included in the appendix are a source of inspiration for the criteria that we use in the Liouville walks context. However, the latter criteria
must be more sophisticated and tailored in a way that makes it
convenient to verify their validity for  quasi-periodic Liouville walks.

}


\begin{theorem}
\label{ThGenErr}
There exist a generic set $\bar \cE \subset \cE$ such that  
for each $\sp  \in \bar \cE$, the following properties are satisfied by the walk $(\barZ_t)_{t\in \N}$ defined in \eqref{DefMCbarZ}:

(a) {\bf (Recurrence)} The walk is recurrent. \medskip

Moreover, there exist strictly increasing sequences $r_k, s_k, t_k$ such that \smallskip

(b) {\bf (Localization)} 
For $T=r_k$

\begin{equation}
\label{EqGenLoc}
\Prob\left(|\barZ_{T}|>(\ln T)^2\right) <  T^{-1/2} \text{ and } \ 
\Var(\barZ_{T})< 2(\ln T)^{4}.
\end{equation}

(c) {\bf (One-sided drift)} \ For $T=s_n$ and some  $\mu_n, \sigma_n$
such that 
$$\liminf_{n\to\infty} \frac{\mu_n(x)}{s_n}>0, \quad
\liminf_{n\to\infty} \frac{\sigma_n(x)}{\sqrt{s_n}}>0$$
we have that for all $z$

\begin{equation}
\Prob\left( \frac{\barZ_{T}-\mu_{n}}{\sigma_n(x)} <  z\right) =\Phi(z).
\end{equation}

(d) {\bf (Two-sided drift)} For $T=t_n$ there exists $b_n$ and $\eps_n$ such that 
$\DS \liminf_{n\to\infty} \frac{b_n}{t_n}>0,$ $\DS \lim_{n\to\infty}\eps_n=0$ and

\begin{equation}
\label{EqGenTwoSideDr}
\Prob \left( |\barZ_T - b_n|<\eps_n b_n\right)>0.1, \quad   
\Prob \left( |\barZ_T + b_n|<\eps_n b_n \right) >0.1.
\end{equation}
\end{theorem}

As mentioned above, the proof of this theorem 
involves similar ideas, albeit simpler, as the ones  used to prove the analogous results 
for Liouville  walks. Therefore some readers may prefer to go through the proofs of Theorem \ref{ThGenErr} before
going over the proofs of the results for the Liouville walks.

We note several interesting consequences of Theorem \ref{ThGenErr}.

We say that the walk satisfies a limit theorem if there 
exist sequences $(b_t)_{t\in \N}$ and $(\sigma_t)_{t\in \N}$ and a proper 
(that is, not concentrated on a single point)
distribution $\cD(\cdot)$ such that for any $z\in \R$
\begin{equation*} 
\lim_{t\to\infty} \Prob (\barZ_t-b_t <  \sigma_t z)=\cD(z). 
\end{equation*}

\begin{corollary} There exist a generic set $\bar \cE \subset \cE$ such that  
for each $\sp  \in \bar \cE$, the following properties are satisfied by the walk $(\barZ_t)_{t\in \N}$:

\label{CrGenNoLT}  
  (a) $\DS \liminf_{t\to \infty} \frac{\ln |\EXP(\barZ_t)|}{\ln t}=0, $
$  \DS \limsup_{t\to \infty} \frac{\ln |\EXP(\barZ_t)|}{\ln t}=1; $
    $$\liminf_{t\to \infty} \frac{\ln (\Var(\barZ_t))}{\ln t}=0, 
  \quad \limsup_{t\to \infty} \frac{\ln (\Var(\barZ_t))}{\ln t}=2. $$
(b) $(\barZ_t)_{t\in \N}$ does {\bf not} satisfy a limit theorem.  
\end{corollary}

\begin{proof}
(a) Follows from parts (b) and (d) of Theorem \ref{ThGenErr}.
Part (b) 
follows exactly as the absence of quenched limit in corollary \ref{cor_nolimit} follows from Theorem \ref{ThFlMV}.
\color{bleu1}  \end{proof} \black

\subsection{Open questions.}
\label{SSOpen}

We close Section \ref{sec.statements} with two 
open questions  about the Liouville one dimensional  walks.

Comparing Theorem \ref{ThAC} with Corollary  \ref{CrNoIM}  leads to the following natural question. 

\begin{question}
Suppose that the walk is symmetric and $\alpha$ is Liouville. 
{ By compactness, the walk 
has at least one stationary measure.} Is the stationary measure unique? Are ergodic stationary measures mixing?
\end{question}

We also note that the maximal growth exponent for the variance of a generic walk 
obtained in Corollary \ref{CrGenNoLT}(a) is optimal, since the variance
of $Z_T$ is at least $\cO(1)$ and at most $T^2.$ However for Liouville  walks we can only show that the
variance grows along a subsequence at a rate that is  not slower than $T^{1-\eps},$ see the discussion after Theorem 
\ref{Th_B}. We believe that the optimal result should be of the same order as for the generic walks. 
Thus we formulate

\begin{conjecture} For generic quasi-periodic symmetric walks, for almost every $x$
$$ \lim\sup_{T\to\infty} \frac{\ln \Var_x(Z_T)}{\ln T}=2.$$
\end{conjecture}

\subsection{Random walks in independent random environments. A brief literature review.}
\label{SSIID}
 A random environment is defined by a sequence  
$\{ \sp_n \}$ of random variables. We can generate it as 
$\sp_n=\fp(\cT^n \omega)$ where $\cT$ is a map of a space $\Omega$
onto itself, and $\fp:\Omega\to [0,1]$ is a 
measurable function. It is usually assumed that $\cT$ preserves a probability measure $\mu$.
The most studied system in this class are iid environments where $\{\sp_n\}$ are independent 
for different $n$.
We consider a random walk in this environment where
the walker moves to the right with probability $\sp_n$ and 
to the left with probability $\sq_n=1-\sp_n$.
The case where $\omega$ is distributed according to $\mu$ is called {\it annealed}, and the case where
$\omega$ is fixed and we wish to obtain the results for $\mu$ almost every $\omega$ is called
{\it quenched}.
Quasi-periodic random walks are examples of random walks in random
environment.


Since there is a vast literature on this subject, we will just briefly discuss the iid
walks here, referring the readers to \cite[Part I]{Z} for more
information.
Let 
$\Delta=\EXP(\ln \sp_n-\ln \sq_n)$. We call the walk {\it symmetric} if $\Delta=0$ and 
{\it asymmetric} otherwise. Let the walk start at the origin, 
and denote by $Z_t$ the position of the walker at time $t$. Here are some results.

According to \cite{Sol}, the walk is recurrent iff it is symmetric. Moreover, if $\Delta>0$, 
 then 
$Z_t\to+\infty$ with probability 1,  and if $\Delta<0$ then $Z_t\to -\infty$ with probability 1.
Surprisingly, in the transient case the walk can escape to infinity
with zero speed or have positive speed and superdiffusive fluctuations.
We also note that, unless the walk satisfies the classical Central Limit Theorem,
there is no quenched limits.\footnote{We refer the readers to \cite{P08, PZ, DG1, ESTZ, PS} for more discussion of the quenched behavior of the walk.}


Thus, the behavior of the random walk in the independent asymmetric
environment can be very different from that of a simple random walk. 
The difference is even more startling in 
the symmetric case where it was shown by Sinai (\cite{S1}) that
$\frac{Z_t}{\ln^2 t}$ converges to
a non-trivial limit (the density of the limit distribution is obtained in \cite{K86}).
The quenched distribution has even stronger localization properties, namely, most of the 
time the walker is localized on the scale $\cO(1)$ 
\cite{G84}. More precisely, given $\eps>0$, we can find an integer 
$N(\eps)$ such that for each $n\in \N$,
for a set of environments $\omega$ of measure more than $1-\eps$, there is a subset 
$\mathfrak{T}_t(\omega)\subset \Z$ of cardinality $N(\eps)$ such that 
$\Prob_\omega(Z_t\in \mathfrak{T}_t(\omega))>1-\eps$.  
This strong localization could be used to show that the symmetric walk does not satisfy 
a quenched limit theorem.
The fact that the fluctuations of the walk in both annealed and quenched case are subpolynomial 
is referred to as {\it Sinai-Golosov localization}.

To understand different behaviors in different regimes in the asymmetric case,
one needs a notion of a {\it trap}. Informally, a trap is a short
segment $I$ where for most of the sites the drift is pointing in the 
direction opposite to the one in which the walker is going. The most convenient way
to do this (\cite{S1}) is in terms of the {\it potential}, defined in
\eqref{Pot}. 
Namely, a segment $I$ is a trap if  the minimal value of the potential 
inside $I$ is much smaller than both boundary values.
The creation of traps  is our main tool for proving the localization of the walker 
in the  Liouville case.

\section{Preliminaries}
\label{ScPrel}
\subsection{Stationary measures.} 
\label{SSIM}
Here we comment on the relation between our findings and the question of the existence of 
absolutely continuous stationary measures for the quasi-periodic walks. We consider a  walk given by a pair $(\a,\fp)$  with $\a \notin \Q$ and $\fp \in C^\infty(\T,(0,1))$. The associated process $(X_t)_{t\in \N}$ has an absolutely continuous stationary measure with density  $\rho(\cdot)$ iff for Haar a.e. $x \in \T$ 
\begin{equation}
\label{EqIM}
\rho(x)=\fp(x-\alpha)\rho(x-\alpha)+\fq(x+\alpha) \rho(x+\alpha). 
\end{equation}
A direct computation shows that the flux
$$ f(x)=\fp(x) \rho(x)-\fq(x+\alpha)\rho(x+\alpha) $$
is constant along the orbit of the rotation and by ergodicity $f(x)\equiv f.$
Now there are two cases 

(I) The walk is not symmetric. We may assume (applying a reflection if necessary)
that 
$$ \int \ln \fp(x) dx>\int \ln \fq(x) dx. $$
In this case we can take (rescaling $\rho$ if necessary)
$f=1$ so that
$$ \rho(x)=\frac{1}{\fp(x)}+\frac{\fq(x+\alpha)}{\fp(x)}\rho(x+\alpha). $$
Iterating further one can obtain following \cite{S2} a smooth solution\footnote{The convergence of the sum defining the solution $\rho$ is guaranteed by the asymmetry condition and the fact that $\a \notin \Q$.} 
\begin{equation}
\label{IMDrift}
\rho(x)=\frac{1}{\fp(x)} \sum_{k=0}^\infty \left(\prod_{j=1}^k
  \frac{\fq(x+j\alpha)}{\fp(x+j\alpha)} \right). 
\end{equation}

(II) The walk is symmetric. In this case using recursive analysis similar to case (I)
one can see that there are no solutions with $f\neq 0$ (see \cite{S2}). In the case $f=0$
the equation reduces to
\begin{equation}
\label{ZeroFlux}
\fp(x) \rho(x)=\fq(x+\alpha) \rho(x+\alpha). 
\end{equation}
Introducing 
$$ g(x)=\fq(x)\rho(x), $$
we see that \eqref{ZeroFlux} reduces to
\begin{equation}
\label{MultCoB}
\frac{\fq(x)}{\fp(x)}=\frac{g(x)}{g(x+\alpha)}. 
\end{equation}
We call  the function $\fp(\cdot)$ such that 
\eqref{MultCoB} has a measurable  solution $g$ a 
{\it (multiplicative) coboundary above $\a$}, and $g$ its  corresponding {\it
  transfer function}\footnote{\red  A coboundary $\fp(\cdot)$ above $\a$ is necessarily symmetric.}. In conclusion, we have the following proposition.

\begin{proposition}
\label{PrIM}
(a) (\cite[Theorem 3.1]{CG}, \cite[Theorem 1.8]{Br1}).
The Markov chain $(X_t)_{t\in \N}$ defined by $(\a,\fp)$ as in \eqref{DefMC_X} has a stationary measure which is absolutely continuous
with respect to the Lebesgue measure iff either the walk is asymmetric or if it is symmetric and $\fp$ is a coboundary above $\a$. 

(b) (\cite[Corollary 6.2]{DG6}) In the case $\a \in \R \setminus \Q$ and $\fp$ is a coboundary, the walk $(Z_t)_{t\in \N}$  defined by $(\a,\fp)$ as in \eqref{DefMCZ} satisfies the CLT.
That is, there exists a constant $D^2>0$ such that for all $x$ all $z$
$$ \Prob_x\left(Z_n\leq  z D\sqrt{n}\right)=\Phi(z). $$
\end{proposition}

{\red Observe that for asymmetric walks, the existence of an absolutely continuous stationary measure for  $(X_t)_{t\in \N}$, or equivalently the existence of a measurable solution to \eqref{IMDrift}, is behind the validity of the quenched central limit of Theorem \ref{ThQLT}, proved in \cite{DG4}.}

We denote by $\cB_\al \subset \cP$ the subset of functions $\fp(\cdot)$ such that 
\eqref{MultCoB} has a smooth solution $g$. Such $\fp$ is called smooth {\it (multiplicative) coboundary above $\a$}.

For every $x\in\T$, denote $\Sigma_x(0)=0$,
\begin{equation} \label{eq.sigma.x} 
\Sigma_x(n)=
\begin{cases}
\sum_{j=1}^{n} \ln \fq(x+j\alpha)-\ln\fp(x+j\alpha), &  n\geq 1, \\
\sum_{j=n+1}^{0} \ln \fp(x+j\alpha) -\ln \fq(x+j\alpha), &  n\leq -1.
\end{cases}
\end{equation}
Notice that if  $\fp\in \cB_\al$, then $\Sigma_x(n)$ has an easy expression
$$
\Sigma_x(n)=\ln g(x+\a)-\ln g(x+(n+1)\a),
$$ 
where $g$ is as in \eqref{MultCoB}.
This behavior of  $\Sigma_x(n)$, as we will
see in \S \ref{subsec_martingale}, renders the walk very similar
to the simple random walk.

In the symmetric case of  Theorem \ref{ThAnn}, a crucial observation
is that for $\a$ Diophantine, any smooth $\fp \in \cP$ is a smooth coboundary
above $\a$. As explained in the introduction, the fact that for a Liouville $\a$, the generic function $\fp \in \cP$ displays very different behaviors of
$\Sigma_x(n)$ at different time scales $n$ and different initial
condition  $x$ is the key behind the phenomena described in Theorems
\ref{ThFlMV}, \ref{Th_B}, and \ref{th.asym} and 
Corollaries \ref {CrNoIM}, \ref{cor_nolimit} and~\ref{cor.asym}.


\subsection{A fundamental martingale.} \label{sec.mart}

{ Until the end of this section we will work with random walks on $\Z$
  in  deterministic elliptic environments as in Definition \ref{Def_set_E}, and
  study the dependence 
of the behavior of such a walk on the potential  defined in \eqref{Pot}.
Notice that for a fixed  triple $(\a,\fp,x)$, the quasi-periodic walk
defined by \eqref{DefMCZ} is equivalent to the walk in a deterministic   
environment defined by $\sp(j)=\fp(x+j\a)$ for all $j\in \Z$.}

{ Recall that we fix a small $\eps_0\leq 0.1$ and let 
\begin{equation}
\cE_{\eps_0}=\{\sp:\Z\to [\eps_0, 1-\eps_0]\}. 
\end{equation}
For $\sp\in \cE_{\eps_0}$, let $\sq(\cdot)=1- \sp(\cdot).$
Consider  the nearest-neighbor random walk $(Z_t)_{t\in \N}$ on $\Z $
defined by the transition probabilities $Z(0)=0$,
\begin{equation}
\label{def_walk_Z}
\Prob(Z_{t+1}=j+1|Z_t=j)=\sp(j), \quad 
\Prob(Z_{t+1}=j-1|Z_t=j)=\sq(j). 
\end{equation}
}

Denote $\Sigma(0)=0$ and
\begin{equation}
\label{EqSigma}
\Sigma(n)=
\begin{cases}
\DS \sum_{j=1}^{n} \ln \sq(j)-\ln\sp(j), &  n\geq 1, \\
\DS \sum_{j=n+1}^{0} \ln \sp(j) -\ln \sq(j), &  n\leq -1.
\end{cases}
\end{equation}
For $j<k$, we use the notation 
$$
\Sigma(j,k):= \Sigma(k)-\Sigma(j).
$$

\label{subsec_martingale} Denote $M(0)=0$, $M(1)=1$, 
\begin{equation}
\label{EqMart}
M(n)=
\begin{cases}
\DS 1+\sum_{k=1}^{n-1} \prod_{j=1}^{k} \frac{\sq(j)}{\sp(j)}, & n\geq 2, \\
\DS -\sum_{k=n+1}^{0} \prod_{j=k}^{0} \frac{\sp(j)}{\sq(j)}, &  n\leq -1.
\end{cases}
\end{equation}
Notice that 
$$
M(n)=
\begin{cases}
\DS \sum_{j=0}^{n-1} e^{\Sigma(j)}, &  n\geq 1, \\
\DS -\sum_{j=n}^{-1} e^{\Sigma(j)}, &  n\leq -1.
\end{cases}
$$

It is straightforward that $M(Z_t)$ is a martingale under $\Prob$. 
The optional stopping theorem for the martingales implies that 
if $Z_{t_0}=n$ with $a<n<b$,
and if $\tau$ is the first time such that $Z_{t_0+\tau}$ the walk reaches either $a$ or $b$,  then 
$M (Z_{\min(|t_0+t|,|t_0+ \tau|)})$ is a martingale under $\Prob$. 
In particular for any $a<n<b$
\begin{equation}
\label{LeftRight}
\Prob (Z_t\text{ reaches }b \text{ before } a|Z_0=n)=\frac{M(n)-M(a)}{M(b)-M(a)}.
\end{equation}
(see e.g. \cite[Theorem 6.4.6]{Du}). 

This formula provides a relation between the sums $\Sigma(j,k)$ and the
behavior of the walk.

We note that \eqref{LeftRight} also holds if $a=-\infty$ or $b=+\infty$ (see e.g. \cite[\S VII.3]{RY}).
In particular, 
\begin{equation}
\label{Rec}
Z_t\text{ is recurrent} \ \Leftrightarrow \ \lim_{n\to-\infty} M(n)=-\infty\text{ and }
 \lim_{n\to+\infty} M(n)=+\infty,
\end{equation}
and 
\begin{equation}
\label{EscRight}
\Prob(\lim_{t\to +\infty} Z_t=+\infty)=\frac{M_-}{M_-+M_+} \ \text{ where } \ M_\pm=\lim_{n\to\pm\infty} |M(n)|.
\end{equation}

\section{Random walks in a deterministic aperiodic medium. Diffusion and localization via 
optional stopping}
\label{ScAbstPot}

For a fixed $\eps_0\leq 0.1$ and $\cE$ as in \eqref{set_E}, fix $\sp\in \cE$ and consider  the  random walk $(Z_t)_{t\in \N}$  defined by \eqref{def_walk_Z}.

In this section we present criteria involving the sums  $\Sigma(n)$ defined in \eqref{EqSigma},
that will be used to guarantee the different behaviors in Theorem A.

Proposition \ref{prop.localization} below gives a criterion for
localization, 
Proposition \ref{prop.drift} for one-sided drift, and  
Proposition \ref{prop.nolimit} for two-sided drift.
The proofs of these propositions will be given in Section \ref{ScCrit}.

\subsection{Localization criterion.} We say that {\bf $\sp$ satisfies
condition $\mathcal C_1( N)$ }
if the following two inequalities hold:
\begin{align*}
\Sigma(N) > N^{1/2}  \tag{$\mathcal C_1a_+ $} ,\\
 \Sigma(-N) > N^{1/2} \tag{$\mathcal C_1a_- $}.
\end{align*}

\begin{proposition} \label{prop.localization} If $\sp$ satisfies
condition $\mathcal C_1(N)$,  then for $T=e^{\sqrt{N}/4}$ we have  

\begin{equation} \label{eq.localization}
\tag{{\bf Localization}} 
\Prob \left( \max_{t\leq T}|Z_{T}|>16(\ln T)^2 \right) <T^{-2}
      \text{ and } \
      \Var(Z_{T})< 300 (\ln T)^4 . 
\end{equation}
\end{proposition} 

{
Condition $(\C_1)$ means that the origin is a sharp local minimum of the potential,
and as explained in the introduction, it implies that the walker spends a lot of time near the origin.}

\subsection{One-sided drift criterion.}
We say that {\bf $\sp$ satisfies condition $\mathcal C_2(N,\eps)$}
for $\varepsilon >0$ if there exist $A>100$ and  $L$ satisfying
$e^{e^A} < L \leq N^{\varepsilon^2}$, $N\leq e^{L^{0.1}}$, 
such that
the following conditions hold:
\begin{align*}\label{Cond2}
& \quad  \Sigma (-L) > \sqrt{L}; \tag{$\mathcal
                  C_2a $} \\
& \quad \Sigma (k, k') < A \ \ \text{for
                  all } k, k'\in[-N,N],
      \ \ k\leq k' ; \tag{$\mathcal C_2b $} \\
& \quad    |\sp(j + kL)-\sp(j)|<N^{-1/\varepsilon^3} 
\ \ \text{for all } (j,k)\in [0,L-1]\times[-N/L,N/L].  \tag{$\mathcal C_2c $} 
\end{align*}

\begin{proposition} \label{prop.drift} 
If $\sp$ 
satisfies condition  $\mathcal C_2(N, \eps)$ for some $\eps>0$, then 
for  $T=N $,   there exist  $\mu > T^{1-\varepsilon}$ and $\sigma$ with  
$\left|  \frac{\ln\sigma}{\ln T}-\frac{1}{2}\right|<\varepsilon $ such that  for every $z\in \R$

\begin{equation} \label{eq.onesided}
\tag{{\bf One-sided drift}}
\left| \Prob \left( \frac{Z_{T}-\mu }{\sigma } <  z\right) - \Phi(z) \right|
<\varepsilon.
\end{equation} 
\end{proposition}

{Conditions $(\C_2(a))$ and $(\C_2(c))$ imply that on a large segment around the origin,
the potential $\Sigma$ is decreasing on scale $L$, while $(\C_2(b))$
means 
that there are no deep wells of the potential (also known as traps)
on smaller scales. Thus, Proposition \ref{prop.drift} confirms the heuristic arguments described in the 
introduction after \eqref{Pot}.}

\smallskip

\subsection{Two-sided drift criterion.} 
We say that {\bf $\sp$ satisfies
condition $\mathcal C_3(N,\eps)$ }  for $\varepsilon >0$
if there exist $A>100$, $e^{e^A}< Q < N^{1/2}$ 
and numbers $u, v, w_\pm, u', v', w'_\pm$ 
such that
$v, v' \in [0.3,0.4]$,
\begin{align*}
0.225 <u <  v-\eps <  w_-  < v < w_+ < v+\eps < 0.5, \\
 0.225 <u' <  v'-\eps < w_-' < v' < w'_+ < v'+\eps < 0.5,
\end{align*} 
and
\begin{align*}
 & \quad \Sigma(v N, w_+ N) > N^{1/2}, \quad 
\Sigma(vN,w_- N) > N^{1/2},  \tag{$\mathcal C_3a $}  \\
& \quad \Sigma(-v'N, -w'_- N) > N^{1/2}, \quad 
\Sigma(-v'N, -w'_+ N) > N^{1/2}, \\
 &  \quad  \Sigma (k,k') < A  \ \ \text{ for }
                    \ \ k, k' \in [-u'N ,  vN ], \quad k' \geq k, \tag{$\mathcal C_3b $}  \\
 &  \quad  \Sigma (k, k') < A  \ \ \text{ for } \
                    \ k, k' \in [-v'N ,  uN ], \quad  k' \leq k, \\
& \quad \Sigma(k) = \BS(k) + B(k), \quad k\in [-v'N,  vN], \tag{$\mathcal C_3c $}
\end{align*}

where $\BS$ and $B$ satisfy 
\begin{equation}\label{C3per}
|\BS(k) - \BS (k+lQ)| <Q^{-1/2}
  \  \text{ for } \    k\in [0,Q], \  l\in [-v'N/Q,  v N/Q] 
\end{equation}
and 
\begin{equation}\label{C3B}
B(k) 
\begin{cases}
=0 &\text{ for } \
                k\in [-u'N, uN], \\
\leq 0 &  \text{ for } k\in  [-v'N ,vN].
\end{cases}
\end{equation}

Figure \ref{figC3} illustrates the behavior of $\Sigma(k)$ for $k\in (-(v'+\eps)N,(v+\eps)N)$ due to $(\mathcal
C_3)$. In the figure we assumed $\bar{\Sigma}\equiv 0$ since \eqref{C3per} and \eqref{C3B}
imply that the effect of $\bar{\Sigma}$ is not important in the behavior of  $\Sigma(k)$.

We note that condition $(\mathcal C_3a)$ implies localization of the
walk around the points $- v'N$ and $vN$ (compare with conditions of Proposition
\ref{prop.localization}). 
{  \begin{figure}[thb]
        \psfrag{-w'+}{$-w'_+N$}
         \psfrag{-w'-}{$-w'_-N$}
 	 \psfrag{-v'}{$-v'N$}
 	 \psfrag{-u'}{$-u' N$}
	 \psfrag{u}{$u N$}
	 \psfrag{w-}{$w_- N$}
	 \psfrag{v}{$vN$}
	 \psfrag{w+}{$w_+ N$}
	 \psfrag{N}{$\sqrt N$}
    	\center{ \includegraphics[height=1.0in]{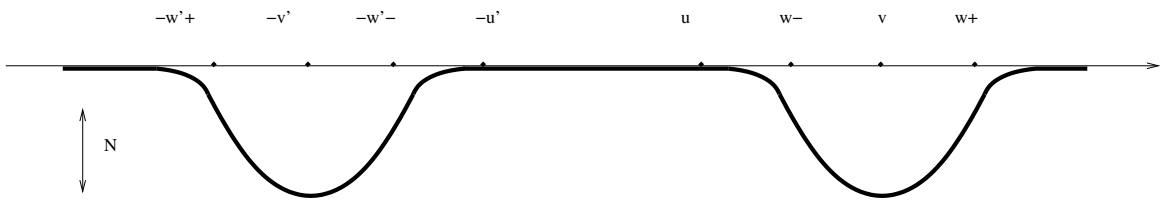}}
		\caption{\small A sketch of the graph of $\Sigma(k)$ under condition $\mathcal C_3(N,\eps)$, in which it is assumed that $\bar{\Sigma}\equiv 0$.}
\label{figC3}
\end{figure}}

Conditions $(\mathcal C_3b_\pm)$ imply that the random walk starting at
zero exits the interval $[- v'N , vN]$ before time $N^5$ with
probability almost one (see Lemma \ref{LmExitExp} below). 

Condition $(\mathcal C_3c)$ compares the walk on $[-v'N ,vN]$
to a $Q$-periodic walk. This condition, combined with $(\mathcal
C_3b_\pm)$, makes sure that for the random walk starting at
zero, both the probability of reaching $- v'N$ before time $N^5$, 
and  the probability of reaching $v N$ before time $N^5$, are
not too small. Since  $(\mathcal C_3a)$ implies localization  around $- v'N$ and $vN$ we get the following.

\begin{proposition} \label{prop.nolimit}
If $\sp \in \mathcal C_3(N,\eps)$, then  for any $T\in[N^{5}, e^{N^{1/4}}]$ we have:
\begin{equation}\label{eq.drift}
\tag{{\bf Two-sided drift}}
\begin{cases}
\Prob\left( Z_T \in [w_-N,w_+N]   \right) > 0.1, \\ 
\Prob\left( Z_T \in [-w'_+N,-w'_-N]   \right) > 0.1.
\end{cases}
\end{equation}
\end{proposition}

\section{Exit time estimates}
\label{ScExit}
In this section we derive the key estimates used in the proof of the
above propositions for $\sp \in \cE$, see \eqref{set_E}. Recall 
the definitions of $\Sigma(n)$ and $M(n)$ given in \eqref{EqSigma}
 and \eqref{EqMart}.

\subsection{Traps.} 
\begin{lemma}
\label{lem_loc}
Suppose that 
for some $N$ condition $(\mathcal Ca_+)$ holds, i.e.,
\begin{equation}
\Sigma(N) > \sqrt{N} \tag{$\mathcal Ca_+ $} .
\end{equation}
Then 
for $T=e^{\sqrt{N}/2}$ we have
$$
\Prob(\max_{t\leq T}  Z_t > N)<  \exp(-\sqrt{N}/2) . 
$$
In the same way, if 
\begin{equation}
\Sigma(-N) > \sqrt{N}  \tag{$\mathcal Ca_- $} ,
\end{equation}
then 
for  $T=e^{\sqrt{N}/2}$  we have
$$
\Prob(\max_{t\leq  T}  Z_t <  -N)<  \exp(-\sqrt{N}/2) . 
$$
Moreover, if for some $N$ both $(\mathcal Ca_+)$ and $(\mathcal Ca_-)$
hold, then for  $T_1=e^{\sqrt{N}/4}$ we have:
$$
\Var(Z_{T_1})<  300 (\ln T_1)^4 .
$$
\end{lemma}
\begin{proof}
Recall the notations and the background material from \S
\ref{subsec_martingale}.   If $(\C a_+)$ holds, then we have $M(N)\geq  e^{\sqrt{N}}$, and  (\ref{LeftRight})
implies that 
$$
\begin{aligned} 
\Prob(Z_t \text{ visits } N \text{ before } 0|Z_0=1)& =\frac{M(1)-M(0)}{M(N)-M(0)}
= \frac{1}{M(N)} \\
&\leq  \exp\left(-\sqrt{N}\right).  
\end{aligned}
$$
Hence, for $L =o(e^{\sqrt{N}})$ we have:
$$ 
\Prob(Z_t\text{ visits } N \;\; \text{ before visiting } 0 \;\; L \text{ times})
\leq 1- \left(1-e^{-\sqrt{N}}\right)^L\leq 2L \exp\left(- \sqrt{N}\right).  
$$
Choosing $L=\exp(\sqrt{N}/2)$, we obtain:
$$ 
\Prob\left(\max_{t \leq  \exp(\sqrt{N}/2)}
Z_t> N\right)<  2\exp(-\sqrt{N}/2).
$$
The  case of $(\C a_-)$ is exactly similar.

Assume now that both $(\C a_+)$ and $(\C a_-)$ hold.
Then for  $T_1=\exp(\sqrt{N}/4)$  we have $N=16(\ln T_1)^2$, and $\DS 
\Var \left(Z_{T_1}\right) \leq  \EXP(Z_{T_1}^2) \leq  N^2 + 2T_1^2
\exp(-\sqrt{N}/2)<  
N^2 + 2 <  300(\ln T_1)^4.$ \indent $\ \hfill $  \color{bleu1}  \end{proof} \black


\subsection{Exit time in the absence of traps.}
 
Let $L\in \NN$. For an arbitrary choice of $k_0 \in (-L,L)$, let $\tau$ be the first time the walk that starts at $k_0$ hits $L$ or $-L$. 
\begin{lemma}\label{LmExitExp}
Suppose that there exist $A>100$ and  $L$ satisfying $e^{e^A} < L$ such that 
for each $k \in[-L,L]$ either $(\mathcal Cb_+)$ or $(\mathcal Cb_-)$ holds: 
\begin{equation}
\Sigma(k,k') < A \text{
  for all } k' \in[-L,L], \quad k' \geq  k;
\tag{$\mathcal Cb_+ $} .
\end{equation}
\begin{equation}
\Sigma (k,k') < A \text{
  for all } k' \in[-L,L], \quad k' \leq  k
\tag{$\mathcal Cb_- $} .
\end{equation}
Then there is a constant $c>0$ such that for $s\in \{1,2,3\}$ we have 
\begin{equation}
\label{M1-3}
\EXP(\tau^s)\leq c e^{sA} L^{2s+1}.
\end{equation}
Moreover $\EXP(\tau) \geq L$ and $\Var(\tau)\geq 1$. 
\end{lemma}
\begin{proof} For every $k \in I=[-L+1,L-1]$,  let $\eta_k$ be the total 
time the walker (starting at $0$) spends at site $k$ before reaching $-L$ or $L$.   Then
$\DS
\tau =  \sum_{k\in  I} \eta_k.
$
Hence, for any $s \in \N$ 
$$
\tau^s  \leq L^s \sum_{k \in I} \eta_k^s.
$$

Thus, it suffices to show that for $s\in \{1,2,3\}$ and for any $k \in I$ 
\begin{equation*}\label{eq.eta} 
\EXP(\eta_k^s)\leq  c e^{sA} L^{s}.
\end{equation*}
For $k \in I$, let  $\bar \eta_k$ be the total 
time a walker starting at site $k$ spends at site $k$ before reaching $-L$ or $L$.

Note that $\bar \eta_k$ has geometric distribution with parameter
$$
r_k=\Prob(Z\text{ starting at $k$ does not return to }k\text{ before exiting }I) .
$$ 
Since $\EXP(\eta_k^s) \leq \EXP(\bar \eta_k^s)$, we finish 
{the proof of \eqref{M1-3}} once we prove the following 

\medskip 

\noindent {\sc Claim. } {If either $(\mathcal Cb_+)$ or
  $(\mathcal Cb_-)$ holds, we have for every $k \in I$
$$
r_k\geq \frac{c}{L e^A}.
$$}

\medskip 
{\bf \color{bleu1} Proof of the Claim.} Fix $k \in[-L,L] $. Observe that since the walk is elliptic ($\sp \in \cE$), then there exists $c>0$ such that 
\begin{align*}
r_k\geq \eps_0 \max \{    &\Prob(Z\text{ visits }L\text{ before }k|Z_0=k+1), \\
&\Prob(Z\text{ visits }(-L)\text{ before }k|Z_0=k-1)  \}.
\end{align*}
Now, if $(\mathcal Cb_+)$ holds, then \eqref{LeftRight} implies
\begin{align*}
\Prob(Z\text{ visits }L\text{ before
 }k|Z_0=k+1) =&\frac{M(k+1)-M(k)}{M(L)-M(k)}= 
\frac{e^{\Sigma(k+1)}}{\sum_{j=k+1}^{L-1}e^{\Sigma(j)} } \\ 
= &\frac{ 1 }{1+ \sum_{j=k+2}^{L-1}e^{\Sigma(k+1, j)} }
> \frac{1}{L e^A}.
\end{align*}
In the same way, if $(\mathcal Cb_+)$ holds, then
$$
\Prob(Z\text{ visits }(-L)\text{ before }k|Z_0=k-1)
> \frac{1}{L e^A},
$$ 
and the claim is proved. 
\color{bleu1}  \end{proof} \black 
Since the walker moves one step at a time, 
$\EXP(\tau)\geq L$. The lower bound on the variance of $\tau$ is
obvious due to the ellipticity condition on the walk.  $\hfill \Box$


\section{Proofs of the criteria} 
\label{ScCrit}
In this section, we prove Propositions \ref{prop.localization}--\ref{prop.nolimit}.

 Proposition \ref{prop.localization} 
follows directly from Lemma
\ref{lem_loc}.

\bigskip \noindent {\bf \color{bleu1}  Proof of Proposition \ref{prop.drift}.}
Fix $N \in \N$ and $\eps>0$ and $A>100$ and  $L$ satisfying
$e^{e^A} < L \leq N^{\varepsilon^2}$, and $N\leq e^{L^{0.1}}$.

{To make the argument easier to follow, we first  consider the periodic case, i.e., we assume that 
 the environment satisfies  
\begin{equation}
\label{PeriodicL}
\sp(k+L)=\sp(k) \text{ for any }k\in \integers.
\end{equation} }
If we run the walk starting from $0$ and stop it at the
time  $\tau$ when it hits either $L$ or $-L$, we get two random
variables: $\tau$ and {$U=Z_\tau$} (thus, $U$ takes values $L$ or $-L$). 
Let us consider iid copies $(\tau_i,U_i)$ of such pairs.
Denote 
$$ 
\hat \mu=\EXP(\tau_i), \quad \hat{V} = \EXP((\tau_i-\hat \mu)^2),
\quad \hat \gamma=\mathbb E(|\tau_i-\hat \mu|^3).
$$
By Lemma  \ref{LmExitExp}
we have the following estimates:
\begin{equation}\label{taui}
L\leq  \hat \mu \leq c e^{A} L^3, \quad 
1\leq\hat V\leq c e^{2A} L^5, \quad 
\hat \gamma  \leq c e^{3A} L^7.
\end{equation}

Note that  
$$
\Prob (U=L)\geq 
1-e^{-0.5 \sqrt{L}} 
$$
by condition $(\mathcal C_2a)$ (cf. the proof of Lemma \ref{lem_loc}).

For $M\in \N$, denote
$$
\Theta_M:=\sum_{i=1}^M \tau_i.
$$   
For $M\leq N^2$ we have that  $\DS Z_{\Theta_M} = \sum_{i=1}^M U_i $ satisfies 
\begin{equation} \label{THM} {\mathbb P}\left(Z_{\Theta_M} = ML\right)\geq (1-e^{-0.5 \sqrt{L}})^{N^2}\geq 1-
e^{-0.1 \sqrt{L}},
\end{equation}  
if $N$ is sufficiently large, since  $N\leq e^{L^{0.1}}$.

Define the stopping time $M_N$ as the first integer such that $\Theta_{M_N}\geq N$. 
By 
\cite[Theorem 1]{lorden}, we have that the "residual lifetime" or "excess over the boundary",
$\Theta_{M_N}- N$ satisfies 
\begin{equation} \label{eqexcess} 
{\mathbb E}(\Theta_{M_N}- N) \leq \hat
V/\hat \mu \leq c e^{2A} L^4,
\end{equation}
from \eqref{taui}. Since $L>e^{e^A}$ Markov inequality implies that
\begin{equation} \label{eqLL} 
{\mathbb P} \left(\Theta_{M_N}\in [N,N+L^{6}]\right)\geq 1-\frac{1}{L}. \end{equation} 

Thus, combining \eqref{THM} and \eqref{eqexcess}, we get
\begin{equation} \label{eqZN} 
{\mathbb P} \left(
|Z_N-M_N L|\leq 2L^{6}\right) \geq 1-\frac{2}{L}.
\end{equation}

By the Berry-Esseen theorem for renewal counting processes
{ \cite[Theorem 2.7.1]{Gu} }
we have
$$
\left| \mathbb P \left(
    \frac{M_N-\frac{N}{\hat \mu}}{\sqrt{\frac{\hat V}{\hat \mu^3} N}} <z
  \right)-\Phi(z)\right| 
<  4\left(\frac{\hat \gamma}{\sqrt{\hat V}}\right)^3
\sqrt{\frac{\hat \mu}{N}}<  \frac{1}{L}
$$ 
if $N$ (and therefore $L$ since  $N\leq e^{L^{0.1}}$) is sufficiently large. Hence 
\begin{equation} \label{eqclt} \left| \mathbb P \left(
    \frac{Z_N - \frac{LN}{\hat \mu}}{L\sqrt{\frac{\hat V}{\hat \mu^3} N}}<z
  \right)-\Phi(z)\right| <  \frac{1}{\sqrt{L}}.
\end{equation}
Since $L\leq N^{\varepsilon^2}$,  (\ref{taui}) implies that
$$
\mu:=\frac{LN}{\hat \mu}>\frac{N}{c e^AL^2}>N^{1-\varepsilon},
$$
and 
$$ 
\sigma: =  L \sqrt{\frac{\hat V}{\hat \mu^3} N}
$$
satisfies $|\ln\sigma /\ln N -1/2|<\varepsilon $.

This completes the proof in the periodic case \eqref{PeriodicL}. 
Now let the periodicity assumption \eqref{PeriodicL} be replaced by the
weaker condition $(\C_2c)$. In this case we 
consider a new periodic environment
$\bar \sp_n$, where $\bar \sp_n=\sp_n$ for each $n\in [0, L]$, and
$\bar \sp_n$ is periodic with period $L$. 
Let $\overline{\Prob}$ denote the corresponding probability for the paths of the walk $(\bar Z_t)_{t\in \N}$ defined by $\bar \sp$. By $(\mathcal C_2c)$, the conditions $(\mathcal C_2a)$ and $(\mathcal C_2b)$ are valid for  $(\bar Z_t)_{t\in \N}$, thus 
the CLT limit \eqref{eqclt} holds for $\bar Z_N$, with the corresponding $\overline{\hat{\mu}}$, $\overline{\hat{V}}$.

Now, for any path of length $N$ for the walks : $(z_1,\ldots,z_N)$, $z_{j+1}=z_j\pm1$,  we have from  $(\C_2c)$ that 
$$ \left|{\Prob}\left(Z_1=z_1,\ldots,Z_N=z_N\right)-\overline{\Prob}\left(\bar Z_1=z_1,\ldots,\bar Z_N=z_N\right)\right|<N^{-1/\eps^2} {\Prob}\left(\bar Z_1=z_1,\ldots,\bar Z_N=z_N\right).$$ 
This shows that $|\overline{\hat{\mu}}-{\hat{\mu}}|<N^{-1}$ and  $|\overline{\hat{V}}-{\hat{V}}|<N^{-1}$, and also that one can replace $\bar Z_N$ by $Z_N$ in \eqref{eqclt}.  Hence, the general case follows from the periodic one. $\hfill \Box$

\ignore{
{\red Reference: 
Berry-Esseen theorem for renewal counting processes

Let $N(T)=\max \{ n: \Theta (n) <  T\}$. 

Theorem: Let $\Theta(n)=\sum_{j=1}^n\tau_j$, where $\tau_j$,
$j=1,2,\dots$, are iid random variables. 
Suppose that $\mu=\EXP(\tau_1)$, $\sigma^2=\Var \tau_1$ and
$\gamma^3=E|\tau_1-\mu|^3$ are all finite. Then
$$
\sup_n\left| P(N(T)<n) -\Phi \left( \frac{(n\mu -T)\sqrt\mu}{\sigma\sqrt{T}}\right)\right|
<  4\left( \frac{\gamma}{\sigma}\right)^3 \sqrt{\frac{\mu}{T}}.
$$
}}


\bigskip \noindent {\bf \color{bleu1} Proof of Proposition \ref{prop.nolimit}.} \ 
We divide the proof into three steps.



\medskip

\noindent {\bf Step 1.  }
For an arbitrary choice of $k\in (-v'N,vN)$, denote by $\tau$ the
exit time from $(-v'N,vN)$ (while starting at $k$).
We need to estimate
$$
\Prob (Z_t \text{ reaches $-v'N$ or $vN$ before time
$N^5$}| Z_0=k) =\Prob(\tau < N^5) .
$$

By Lemma
\ref{LmExitExp}, under condition $(\mathcal C_3b)$ there exists
$c>0$ such that for any $k\in (-v'N,vN)$ we have:
$\EXP (\tau) < ce^A N^3$. Then
$\Prob(\tau > N^4) N^4 < \EXP (\tau) <ce^A N^3$, so
$$
\Prob(\tau > N^4) < ce^A / N .
$$
This implies
$$
\Prob(\tau > N^5) < (ce^A / N )^N <e^{-N}.
$$
Hence,
$$
\Prob (Z_t \text{ reaches $-v'N$ or $vN$ before time $N^5$})
> 1-e^{-N}.
$$

\medskip

\noindent {\bf Step 2.  } We have the following two inequalities: 
\begin{align*} \Prob(Z_t  \text{ visits } -v'N \text{ before visiting } vN ) 
&\leq 0.89, \\
\Prob(Z_t  \text{ visits } vN \text{ before visiting } -v'N ) 
&\leq 0.89.
\end{align*} 
We prove the first estimate, the second one can be proved
similarly. 
By \eqref{LeftRight}
$$
\Prob(Z  \text{ visits } (-v'N)\text{ before visiting }  vN ) =
{\frac{M(vN)}{M(vN)-M(-v'N)}}.
$$
Using  ($\mathcal C_3c$), we get:
$$
M(vN)=
\sum_{j=1}^{vN}e^{\Sigma (j)} \leq
\sum_{j=1}^{vN} e^{\BS (j)} \leq 
\sum_{l=1}^{ vN/Q+1 } \sum_{j=1}^{Q} e^{\BS (j+(l-1)Q)} 
\leq \left(\frac{vN}{Q}+1\right)M(Q) (1+2Q^{-1/2}). $$
In the same way,
$$
M(uN)=
\sum_{j=1}^{uN}e^{\Sigma (j)} =
\sum_{j=1}^{uN} e^{\BS (j)} \geq
\sum_{l=1}^{ uN/Q } \sum_{j=1}^{Q} e^{\BS (j+(l-1)Q)} 
\geq  \frac{uN}{Q}M(Q) (1-2Q^{-1/2}). $$

Hence

$$M(vN)\geq M(uN)\geq \frac{uN}{Q}M(Q) (1-2Q^{-1/2}).$$

Similarly 
$$M(-v'N)\leq M(-u'N) \leq -\left(\frac{u'N}{Q}-1\right)M(Q)(1-2Q^{-1/2}).$$

Hence,  
$$
\frac{M(vN)-M(x) }{M(vN)-M(-v'N)} <  \frac{v}{u+u'}+0.01 < 
\frac{0.4}{0.45}+0.01 < 0.89.
$$

\medskip

\noindent {\bf Step 3.  }  By Step 1, with probability $1-e^{-N}$, the
walk starting at 0 reaches 
either $vN$  or $-vN'$ before time $N^5$. By Step 2, it reaches $vN$
before time $N^5$ with probability larger than $0.1$.  
The first part of $(\mathcal C_3a)$ states that
$\Sigma(vN,w_+N) > N^{1/2}$ 
and $\Sigma(vN,w_-N) > N^{1/2 }$. Under this condition, 
Lemma \ref{lem_loc} implies that the walk starting
at $vN$ satisfies  
$$
\Prob\left( Z_T \in [w_-N, w_+N] \right) > 0.1
$$
for all $T\in [N^5, e^{N^{1/4}}]$, 
which implies the desired result. The same argument holds for $-v'N$.  $\hfill \Box$


\section{Quasi-periodic environments}
\label{ScConstr}

In this section we return to the study of quasi-periodic random walks. Fix a Liouville number $\a \in  \R \setminus \Q$ and  $\fp \in \cP$, and for $x \in \T$ consider the 
corresponding environment: 
\begin{equation*}
\label{QPEnv}
\sp(j): =\fp(x+j\alpha),
\end{equation*}
as well as the random walk  $(Z_t)_{t\in \N}$ on it, defined by (\ref{DefMCZ}). 

It will be convenient to reformulate the main conditions $\mathcal
C_j$, $j=1,2,3$, in this new context. 
First of all, notice that
condition $\mathcal C_1=\mathcal C_1(N)$ does not depend on $\eps$,
while the other two conditions do. For the uniformity of
notations, we formally include an $\eps$ in all the
three conditions. 
We say that 
$$
x\in \mathcal C_j(\fp, N,\eps)\ \text{ if and only if } 
\sp \in \mathcal C_j(N,\eps), \quad j=1,2,3.
$$

The goal of this section is to prove the following statement.

\begin{theorem}
\label{Cond_to_behaviors}
For any Liouville $\alpha$ there exists a dense $G_\delta$ set 
$\cR \subset \cP$ with the following property: for any $\fp\in \cR$,  
for almost every $x\in \T$,  there are strictly increasing sequences 
of numbers $N_{j,n}$, such that for all $j=1,2,3$, $n\in \NN$ we have  
$$
x\in \mathcal C_j(\fp, N_{j,n}, 1/ n) .  
$$ 
\end{theorem}
By the results of Propositions
\ref{prop.localization}--\ref{prop.nolimit}, 
this will suffice, 
to prove Theorem~\ref{ThFlMV}, see \S \ref{Sec_A} for details.

\subsection{The $G^\delta$ argument.} 
\label{SSRedCob}
In \S \ref{SSIM} we denoted by $\cB_\al \subset \cP$ the set of  
(multiplicative) coboundaries, i.e., 
functions $\bar \fp(x)$ such that 
\eqref{MultCoB} has a solution with a smooth transfer function $g$. 
We noticed after formula \eqref{eq.sigma.x} that
if  $\bar \fp\in \cB_\al$, 
then
$$
\Sigma_x(n)= \ln g(x+\a)- \ln g(x+(n+1)\a).
$$
In particular, for all $n$, $|\Sigma_x(n)|$ is bounded by a constant
independent of $n$, hence none of the criteria from the previous section
holds for $\bar \fp\in \cB_\al.$
The advantage of coboundaries however, is
that due to the simple formulas for their ergodic sums one can develop a reasonable 
perturbation theory and check our criteria for perturbations $\bar\fp$ of $\fp$ on appropriate
subsets of $\T.$


Thus, we start by proving that the set of  coboundaries
are dense in $\mathcal P$. 

\begin{lemma} \label{coboundaries} For any $\alpha\in \R\setminus
 \QQ$, the set $\cB_\al$ of smooth multiplicative coboundaries
is dense in $\cP$ for the $C^\infty$ topology.

Similarly, any  $\fp \in C^\infty(\Tor,(0,1)) \cap \cP^c$ can be approached in the $C^\infty$ topology by $\bar \fp$ such that for $\bar \fq=1-\bar \fp$ and some constant $c\neq 0$, there exists $\psi \in  C^\infty(\T,\R)$ such that 
$$\ln \bar \fq(x) -\ln \bar \fp(x) = c+\psi(x+\a)-\psi(x).$$
\end{lemma}

\begin{proof}
Let $\alpha\in \R\setminus \QQ$. By truncating the Fourier series of $\ln
(\fq/\fp)$ it is possible to approximate it in the $C^\infty$ topology by
coboundaries of the form $\psi(\cdot)-\psi(\cdot+\a)$ where $\psi \in
C^\infty(\T,\R)$. Hence $F(\cdot)= g(\cdot)/g(\cdot+\a)$ where $g=e^{\psi}$
can be made arbitrary close to $\fq/\fp$. Now define $\bar \fp = 1/(1+F)$ and observe that
$\bar \fp$ approximates $\fp$ and $\bar \fp \in \cB_\al$. 

 The case $\fp \in C^\infty(\Tor,(0,1)) \cap \cP^c$ is treated similarly except that $\ln
(\fq/\fp)$ is approximated by  $c+ \psi(\cdot)-\psi(\cdot+\a)$ with $c\neq 0$. 
\color{bleu1}  \end{proof} \black

To prove Theorem \ref{Cond_to_behaviors}, we will construct 
explicit sequences of  
functions $e_n$, numbers $N_{j,n}$ and sets $U_{j,n}$ such that  
the following statement holds true:
\begin{proposition}\label{PropCond}
For any Liouville  number $\alpha$ there exists
\begin{itemize}
\item A strictly increasing sequence of integers $q_n$, 
\item  An explicit sequence of {$C^\infty$}
functions $e_n$ satisfying 
$|e_n|_{C^{n-1}}<1/n$ (see \S \ref{Sec_en}),
 
\item  For every $j \in \{1,2,3\}$,  a sequence of numbers $N_{j,n}$, 

\item For every $j \in \{1,2,3\}$, a sequence of sets 
$\DS U_{j,n}=\bigcup_{i
    \in [0,q_n-1]} (I_{j,n} +i/q_n) \subset \T$, where $I_{j,n}$ is
  an interval in $[0,1/q_n]$ of size $|I_{j,n}|> \frac{0.01}{q_n}$
\end{itemize}
with the following property. For every $\bar \fp\in \cB_\al$, for
every sufficiently large $n$ 
$$
U_{j,n} \subset \C_j\left(\bar\fp+e_n, N_{j,n}, \frac{1}{n}\right) .  
$$
\end{proposition}


Before proving  Proposition \ref{PropCond} we show how it
implies  Theorem \ref{Cond_to_behaviors}.

\bigskip \noindent {\bf \color{bleu1}  Proof of Theorem \ref{Cond_to_behaviors}.} \ 
Fix a  Liouville  $\al\in\R\setminus\QQ$. Fix any $j \in \{1,2,3\}$. 
{Let $U_{j,n}$ and  $N_{j,n}$ be as in  Proposition \ref{PropCond}. Denote}
$$
{\cR}_{j,n} = \{ \fp \in  \mathcal P \mid  U_{j,n} 
\subset  \mathcal C_j (\fp, N_{j,n}, 1/ n)\}.
$$
By the definitions of the conditions $\mathcal C_j$, the sets $
{\cR}_{j,n} $ are open.   Lemma \ref{coboundaries} and Proposition \ref{PropCond} show that for any 
$m\in \NN$ the set 
 $\bigcup_{n\geq m} \cR_{j,n}$
 is dense.
Hence the set
$$
\cR_j=\bigcap_{m\in \NN} \bigcup_{n\geq m} \cR_{j,n}
$$
is a dense $G_\delta$  set (in $C^r$ topology for any $r\in \naturals$).

Observe now that for $\fp \in \cR_j$ we have that there exists a
strictly increasing sequence $l_n$ such that 
$\fp \in \cR_{j,l_n}$. Recall that $U_{j,l_n}$ has Lebesgue measure
larger than $0.01$ for every $l_n$. Moreover, up to extracting a
subsequence we may assume that  $q_{l_{n+1}} \gg {q_{l_n}}$, so  that 
 $$
\lambda \left(U_{j,l_n} \cap \bigcap_{i=1}^{n-1} U_{j,l_i}^c\right)
\geq \frac{1}{2} 
\lambda (U_{j,l_n})\cdot \lambda \left( \bigcap_{i=1}^{n-1}
  U_{j,l_i}^c\right).
$$ 
Now, an enhanced version of the Borel-Cantelli Lemma 
(see \cite[Chapter IV]{Neveu}) states that if
events $C_n$ are such that for each $k\geq 1$ 
$$\sum_{n=k}^\infty \mathbb P\left(C_n \left| \bigcap_{j=k}^{n-1} C_j^c\right.\right)=+\infty,$$ 
then with probability 1, infinitely many of those events occur. 
We thus conclude that 
a.e. $x$ belongs to infinitely many $U_{j,l_n}$, thus to infinitely
many $\mathcal C_j(\fp, N_{j,l_n}, 1/n)$.  In conclusion, the set  
$\DS \cR=\bigcap_{j=1}^{3} \cR_j$ satisfies the property required in Theorem \ref{Cond_to_behaviors}. $\hfill \Box$

The rest of Section \ref{ScConstr} is devoted to the proof of Proposition  \ref{PropCond}.

\subsection{Perturbation of a smooth coboundary. The Main construction.}\label{Sec_main_constr}
\subsubsection {\bf Coboundaries.} \label{sec.cob}
Given a smooth coboundary  $\bar \fp\in \cB_\al $ with a transfer function $g(x)$,
for $M\in \NN$, let $\BS_x(M)$ 
be the ergodic sum of $\bar \fp$ defined by formula (\ref{eq.sigma.x}) with
$\fp$ replaced by $\bar \fp$. Thus, for all  $k\leq
k'$ and all $M\in \NN$ we denote:
\begin{equation*}\label{eqA}
\begin{aligned}
&\BS_x(M) =\ln g(x+\a)-\ln g(x+(M+1)\alpha), \\
&\BS_x(k, k') =\ln g(x+(k + 1)\a)-\ln g(x+(k' + 1)\alpha), \\
&A:=\ln  \|g\|+\ln \|1/g\|.
\end{aligned}
\end{equation*}
Then, 
\begin{equation*} \label{sigbar}
 \forall x\in \Tor, \forall M\in \NN
: |\BS_x(M)|\leq A, \quad |\BS_x(k, k')|\leq A, 
\end{equation*}

Moreover, for any smooth coboundary  $\bar \fp$, there exists  $0<\kappa \leq 1/2$ such that 
\begin{equation*}
\label{EllConst}
\kappa \leq \bar\fp(x) \leq 1-\kappa \quad \text{for all }\quad x\in \Tor. 
\end{equation*}
Define
\begin{equation}\label{eq_K}
K(x)=\frac{1}{\bar\fp(x  )}+\frac{1}{\bar\fq(x )},
\end{equation}
and observe that  \begin{equation*}\label{eq_KK}
K(x)\in (2,2/\kappa]. \end{equation*}

\subsubsection {{\bf The sequences} $q_n$ {\bf and} $N_{j,n}$}
Given a Liouville  number $\al\in\R\setminus\QQ$, let 
$q_n$ be a sequence of integers satisfying
\begin{equation}\label{def_qn}
\eta_n:= |q_n\al| < q_n^{-n^n},
\end{equation}
where $|\cdot|$ denotes the closest distance to integers.
Moreover, for each $n$ we will need to choose 
$q_n$ sufficiently large for our arguments to hold.

Denote the integer closest to $q_n\al$ by $s_n$. In the constructions below
we assume that  $q_n\al>s_n$ for all $n$. If $q_n\al < s_n$, 
the arguments are the same up to a suitable change of signs.

Denote 
\begin{equation}\label{def_Tn}
\begin{aligned}
&N_{n}:=\left[  (q_n\eta_n)^{-1} \right]q_n, \\
&N_{1,n}:=\left[N_n/20 \right],\quad N_{2,n}:= q_n^{n^5}, \quad N_{3,n}: = N_{n},
\end{aligned}
\end{equation}
were $[a]$ stands for the integer part of $a$.

We make the following useful observation on the combinatorics of the irrational rotation $R_\a$ on the circle. The orbit of any fixed point $x$ of the circle  under the
rotation by $\alpha$ on $\T$
is essentially distributed in the following way. The points $x,x+\alpha,
x+2\alpha,\dots x+(q_n-1)\alpha$ are very close (closer than $\eta_n$) to $x,x+\frac{p_n}{q_n},
x+2\frac{p_n}{q_n},\dots x+(q_n-1)\frac{p_n}{q_n}$.   Hence, there will be one point of this $q_n$ piece of orbit in each {\it basic interval} $[k/q_n, (k+1)/q_n]$, $k=0,\dots q_n-1$. Moreover, the first return of $x$ to its basic interval will be shifted by $\eta_n$. The next return will thus be shifted by one more $\eta_n$. Finally, the orbit $x,x+\alpha, x+2\alpha,\dots, x+N_{n}\alpha$ will form an 
 $\eta_n$-grid inside each basic interval, and thus in the whole circle.

\subsubsection  {{\bf The functions} $e_n$} \label{Sec_en} 
In this section, the names of functions with the
shortest 
period 1 are
marked with a tilde, while $\frac{1}{q_n}$-periodic functions have no
tilde in their name. 

Let  $\widetilde e_n(x)\in
C^{\infty}$ be a 1-periodic function satisfying 
$\int_\Tor \widetilde e_n(x) dx=0$ and such that

$$
\widetilde e_n(x)=
\begin{cases} 
\sin 8\pi x  & \text{for } x\in  [-\frac12+ \frac{1}{n^2} ,-\frac38- \frac{1}{n^2} ]
\cup [-\frac38+ \frac{1}{n^2}, -\frac14-\frac{1}{n^2} ] \cup \\
    & \quad \quad  \  \quad[\frac14+ \frac{1}{n^2}, \frac38-
    \frac{1}{n^2} ]
\cup[\frac38+ \frac{1}{n^2} ,\frac12-\frac{1}{n^2} ] 
  , \\
0  & \text{for } x\in  [-\frac14, \frac14], \\
\text{increasing } & \text{on the intervals \ } [-\frac{1}{2},-\frac{1}{2}+ \frac{1}{n^2}], [-\frac{1}{4}- \frac{1}{n^2},-\frac{1}{4}], [\frac{1}{4},\frac{1}{4}+ \frac{1}{n^2}], [\frac{1}{2}- \frac{1}{n^2},\frac{1}{2}], \\
\text{decreasing } & \text{on the intervals \ } [-\frac38- \frac{1}{n^2},-\frac38+ \frac{1}{n^2}], 
[\frac38- \frac{1}{n^2},\frac38+ \frac{1}{n^2}],\\ 0 \text{ and } 
\infty\text{-flat}
&\text{at } x=\pm\frac38. \end{cases}
$$
Observe that $\widetilde e_n$ is also flat at $\pm\frac14$ since it is smooth.
Figure \ref{function_e} represents the function~
$\widetilde e_n$.

{  \begin{figure}[thb] 
        \psfrag{I1}{$I_3$}
 	 \psfrag{I2}{ $I_2$}
 	 \psfrag{I3}{ $\ \, I_1$}
	 \psfrag{1}{\small $1$}
	 \psfrag{12}{\small $\frac12$}
	 \psfrag{-12}{\small $\! \! -\frac12$}
	 \psfrag{14}{\small $\frac14$}
 	 \psfrag{38}{\small $\frac38$}
    	\center{ \includegraphics[height=1.6in]{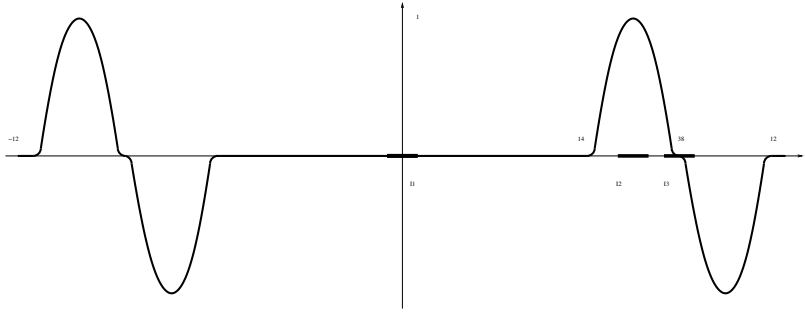}}
    	\caption{\small Graph of $\widetilde e_n(x)$. The intervals $I_1$, $I_2$ and $I_3$ are the sets for which the conditions $\mathcal C_1$, $\mathcal C_2$ and $\mathcal C_3$ hold (Lemmas \ref{lem_cond1}--\ref{lem_cond3}).}
    	\label{function_e}
\end{figure}}

The idea is to perturb a given smooth coboundary  $\bar \fp$ 
by a function of the form $q_n^{-n}\widetilde e_n(q_nx) $   
to produce the desired behavior of the walk. For each $n$ we will
choose $q_n$ satisfying (\ref{def_qn}) and sufficiently large. 
In particular, although the $C^k$ norms of 
$\widetilde e_n$ may grow fast as $n$ grows, we can still guarantee that 
$\|q_n^{-n}\widetilde e_n(q_nx)\|_{C^{n-1}}
<  \frac{1}{n}$. 

 A small problem is that the
perturbed function  
$\fp(x) = \bar \fp (x)+ q_n^{-n}\widetilde e_n(q_nx)$ may not satisfy
the symmetry
condition (\ref{eq_sym}).  Below we modify $\widetilde e_n(x)$ in order
to assure condition (\ref{eq_sym}) for $\fp$. 
Let $\widetilde e_n^+(x)$ and $\widetilde e_n^-(x)$ be the positive and
negative  parts of $\widetilde e_n(x)$:
$$
\widetilde e_n^+(x) = 
\begin{cases}
\widetilde e_n(x) \quad  \text{ if } \widetilde  e_n(x) \geq 0, \\
0, \quad  \hspace{0.55cm}  \text{ otherwise,}
\end{cases}  \quad
\widetilde e_n^-(x) = 
\begin{cases}
-\widetilde e_n(x) \quad  \text{ if } \widetilde  e_n(x) < 0, \\
0, \quad  \quad \hspace{0.6cm}  \text{otherwise. }
\end{cases}
$$
For $\delta \in [-1,1]$, define $\widetilde e_{n,\delta}(x)$:
\begin{equation}\label{def_endn}
\widetilde e_{n,\delta}(x) = 
\begin{cases}
\widetilde  e_n(x) + \delta \widetilde e_n^+(x)  \text{ if } \delta \in [0,1],\\
\widetilde  e_n(x)+\delta \widetilde e_n^-(x)  \text{ if } \delta \in [-1,0).
\end{cases}
\end{equation}
{Note that, since $\widetilde e_n$ is flat at $\pm
  \frac38$ and $\pm\frac14$, where it actually changes the sign, 
the functions  $\widetilde e_{n,\delta}$ are also smooth.  
This is the only reason why we need $\widetilde e_n$ to be flat at those points.}

The following lemma introduces the function $e_n$ that will be the main building block of our
construction. 

\begin{lemma}\label{Lem_est_cn} Given any  $\bar \fp \in \cB_\al$, if $q_n$ is sufficiently large, 
there exists $\delta_n \in [-\frac1n,\frac1n] $ satisfying
\begin{equation}\label{def_pn}
\fp_n(x):= \bar \fp (x) + e_n(x) \in \mathcal P \quad  \text{ with } \quad 
e_n(x) = q_n^{-n} \widetilde e_{n,\delta_n}(q_n x) .
\end{equation}
\end{lemma}

\begin{proof}
In this proof, we will use the notation  
$$
e_{n,\delta}(x): = q_n^{-n} \widetilde e_{n,\delta}(q_n x).
$$

We are looking for $\delta \in [-1/n,1/n]$ such that $\fp_n(x)$ satisfies the symmetry
condition (\ref{eq_sym}), i.e.,
$$
I_{n,\delta}:=\int_{\T} \ln (\fp_n(x))-\ln (1-\fp_n(x))dx= 
\int_{\T} \ln \left(1+ \frac{ e_{n,\delta}(x)}{\bar \fp(x)} \right)- 
\ln \left(1 -\frac{ e_{n,\delta}(x)}{1-\bar \fp(x)} \right)dx=0
$$ 
We will approximate $I_{n,\delta}$ with 
$$J_{n,\delta} := \int_{\T}  e_{n,\delta}(x) K(x)dx$$
where $K$ is given by \eqref{eq_K}.

\noindent{\sc Claim.} There exists a constant $c>0$ that does not depend on $n$ or $\delta$ such that  
\begin{align}
\label{ll1} |I_{n,\delta}-J_{n,\delta}|&<cq_n^{-2n}\\
\label{ll2} |J_{n,0}|&<c\frac{1}{n^2} q_n^{-n}\\\label{ll3} 
\text{For } \delta>0,   \quad J_{n,\delta} -J_{n,0} &>c \delta q_n^{-n}\\
\label{ll4} \text{For } \delta<0,   \quad J_{n,\delta} -J_{n,0} &<-c \delta q_n^{-n}
\end{align}

\medskip 
From the continuity of $I_{n,\delta}$ and $J_{n,\delta}$ in $\delta$, it follows directly from the claim that there exists $\delta \in (-1/n,1/n)$ such that   $I_{n,\delta}=0$. 

\medskip 

\noindent{\sc Proof of the Claim.}
\eqref{ll1} follows from the fact that    $\max |e_{n,\delta}(x)|\leq
2q_n^{-n}$. 
\eqref{ll2} follows from the fact that the average of $ \widetilde e_n$ is zero and from the fact that $K$ is almost constant on intervals of size $1/q_n$. 
As for \eqref{ll3}, it follows from the fact that $K>2$, and that the average of $ \widetilde e_n^+$ is larger than some positive constant independent of $n$.  \eqref{ll4} is proved similarly. \hfill $\Box$

Lemma \ref{Lem_est_cn} is thus proved. \color{bleu1}  \end{proof} \black

\subsubsection {\bf The sets $U_{j,n}$}
Consider the following subintervals of $[0,1]$: 
$$
I:=\left(-\frac1{200},\frac1{200} \right),
$$
$$
I_{1}:=\frac3{8}+I, \quad
I_{2}:=\frac5{16}+I, \quad I_{3}:=I, 
$$
$$
I'_{1}:=-\frac3{8}+I, \quad I'_{2}:=-\frac7{16}+ I, 
$$
and let  $I_{j,n}=I_j/q_n$, $I'_{j,n}=I_j'/q_n$ for $j=1,2,3$,
\begin{equation}\label{int_U3}
U_{j,n}= \bigcup_{k=0}^{q_n-1} \left( I_{j,n}\cup I'_{j,n}
  +\frac{k}{q_n}\right) , \quad 
j=1,2,\quad
U_{3,n} = \bigcup_{k=0}^{q_n-1} I_{3,n} +\frac{k}{q_n}.
\end{equation}
Notice the total measure of $U_{j,n}$: $|U_{j,n}|=0.02$,  and
$|U_{3,n}|=0.01$.

\subsection{Estimates of ergodic sums.}
\label{SSErgSum}

Recall that  $\fp(x)=\bar\fp(x)+e_n(x)$  (see \eqref{def_pn}), 
and that $N_{j,n}$ is defined by  \eqref{def_Tn}. 

For $\fp(\cdot)$ and $\bar\fp(\cdot)$, we denote by $\Sigma_x(n)$ and $\BS_x(n)$ the potential functions defined in \eqref{eq.sigma.x}.  

{The next statement represents the sums $\Sigma_x(M)$ for a
  large $M$ in the form  \\$\Sigma_x(M)=${\it Main term}$(M)+${\it
    Rest}$(M)$. 
Notice that the {\it Rest}$(M)$ may not be small, see \eqref{est_int}
and \eqref{est_int-}. Nevertheless, these estimates are sufficient for
the proofs in the next subsection. There we will show that, for certain
  values of $x$, 
  the {\it Rest}$(M)$ is asymptotically smaller than
  the {\it Main term}$(M)$ provided that $M$ and $q_n$ are sufficiently large.}

\begin{proposition}[Main technical lemma]\label{prop_estimates} 
Given $\al\in\R\setminus
\QQ$ and a smooth coboundary  
$\bar \fp \in \cB_\al $, let   $A,\kappa$, and $K(x)$ be as in \S
\ref{sec.cob}, and let $\widehat K= \int_\T K(x)dx$. 
Then for all $x\in \Tor$
and all $M\in [0,N_n]$ there exist functions $R(x,M)$ and $R'(x,M)$ satisfying
\begin{equation} \label{eq.R} 
|R(x,M)|,|R'(x,M)|\leq 4\kappa^{-2} M q_n^{-2n}
\end{equation}
such that
\begin{align}\label{main_est}
\Sigma_x(M)&=  - \sum_{m=1}^{M} e_n(x+m\a )K(x+m\a)
+ \BS_x(M) +R(x,M), \\
\label{main_est-}
\Sigma_x(-M)&= \sum_{m=-M+1}^{0} e_n(x+m\al )K(x+m\a)
+ \BS_x(-M)
+R'(x,M) ,
\end{align}
where
{
\begin{equation}\label{est_int}
\sum_{m=1}^{M} e_n(x+m\al ) K(x+m\a) = \widehat K q_n^{-n} N_n \int_{q_n x}^{q_n
  x + M/N_n} \widetilde e_{n,\delta_n}(t) \, dt +M o\left(q_n^{-n}\right), 
\end{equation}
\begin{equation}\label{est_int-}
\sum_{m=-M+1}^{0} e_n(x+m\al ) K(x+m\a) = \widehat K q_n^{-n} N_n \int_{q_n x - M/N_n}^{q_n
  x } \widetilde e_{n,\delta_n}(t) \, dt +M o\left(q_n^{-n}\right), 
\end{equation}
}
Moreover,
\begin{equation} \label{est_cond_b} 
\begin{aligned}
e_n(x+m\alpha)\geq 0 \ \text{ for all } m=k,\dots, k' \ \Rightarrow 
\Sigma_x(k, k')\leq \BS_x(k, k')\leq A  
\end{aligned}
\end{equation}
\begin{equation} \label{est_cond_b-} 
\begin{aligned}
e_n(x+m\alpha)\leq 0 \ \text{ for all } m=k',\dots , k \ \Rightarrow
\Sigma_x(k, k')\leq \BS_x(k, k')\leq A . 
\end{aligned}
\end{equation}


\end{proposition}

The statement of this lemma covers  several different situations that
will be useful in 
checking all the conditions $\C_j$. 

\medskip 

\begin{proof} Recall the notations:  $\fp(x)=\bar\fp(x)+e_n(x)$,
  $K=\frac{1}{\bar\fp(x)}+\frac{1}{\bar\fq(x)}$. Omitting the
  argument $x$, we can write:
$$
\begin{aligned}
&\ln (1-\fp)-\ln \fp = \ln(1-\bar \fp -e_n)  - \ln(\bar \fp +e_n) = \ln (1-\bar
\fp)-\ln \bar \fp + \ln (1-\frac{e_n}{1-\bar \fp})\\ 
&-\ln (1+\frac{e_n}{\bar \fp})= 
\ln \bar \fq -\ln \bar \fp - e_n \left(\frac{1}{1-\bar \fp}+ \frac{1}{\bar \fp}\right)
+ r_n
= \ln \bar \fq -\ln \bar \fp - e_n K + r_n,
\end{aligned}
$$ 
where $|r_n|\leq 2 e_n^2(\bar q^{-2}+\bar p^{-2})\leq \|K^2\| \leq 2
\kappa^{-2}$. The estimate of $|r_n|$
follows from the Taylor expansion of $\ln(1+y)$ for small $y$.

Now we estimate $\Sigma_x(M)$ for any $x\in \T$ and any $M>0$:
\begin{equation}\label{eq_long}
\Sigma_x(M)= \sum_{m=1}^{M} \ln (\fq_n(x+m\al ))-\ln (\fp_n(x+m\al ))= 
\end{equation}
$$
\sum_{m=1}^{M} \ln (\bar\fq(x+m\al ))-\ln (\bar\fp (x+m\al )) - 
\sum_{m=1}^{M} e_n(x+m\al ) K(x+m\a)
+R(x,M)
$$
and     
$$
|R(x,M)|\leq \left|\sum_{m=0}^{M} e_n^2(x+m\al )
\left(\bar\fq^{-2}(x+m\al  )+\bar\fp^{-2}(x+m\al ) \right) \right| 
\leq 4 \kappa^{-2} M q_n^{-2n},
$$
since $\|e_n\|\leq 2 q_n^{-n}$ and $\bar\fp, \bar\fq \in [\kappa, 1-\kappa]$.
This gives
 \eqref{main_est} and \eqref{eq.R}. The proof of \eqref{main_est-} is similar.

Let us prove \eqref{est_int}. {
By (\ref{def_qn}) and (\ref{def_Tn}), we have $\eta_n=|q_n\al| <
q_{n}^{-n^n}$, and $\eta_n \approx N_{n}^{-1}$. By definition,
$e_n(x+\a)=q_n^{-n} \widetilde e_n(q_n(x+\a))= q_n^{-n} \widetilde
e_n(q_nx +\eta_n ) $. 
For each $j\leq q_n$ we have: 
$$
|e_n(x+j\a) - e_n(x)|= q_n^{-n} |\widetilde
e_n(q_nx+j\eta_n)-\widetilde e_n(q_nx))| \leq q_n^{-n}\eta_n q_n
\text{max}_\T \,|\widetilde e'_n(x)| = o(q_n^{-n}), 
$$
since $\eta_n=q_n^{-n^n}$ by \eqref{def_qn},
and  $\widetilde e_n(x)$ does not depend on $q_n$, see
\S \ref{Sec_en} for the definition of $\widetilde e_n(x)$.
Since $\a$ is close to $p_n/q_n$, we get: 
\begin{align*} \sum_{m=1}^{q_n} e_n(x+m\a)K(x+m\a) &=  
(e_n(x)+o\left(q_n^{-n}\right)) \sum_{m=1}^{q_n} K(x+m\a) \\
&= \widehat K  q_n (e_n(x) + o\left(q_n^{-n}\right)).
\end{align*}
Hence, for $M \gg q_n$ we have
\begin{align*}
\sum_{m=1}^{M} e_n(x+m\a) K(x+m\a)&=
\widehat K q_n  \sum_{m=1}^{M/q_n} e_n(x+mq_n\al ) +M \, o(q_n^{-n}) \\
=
\widehat K  q_n N_{n} \sum_{m=1}^{M/q_n} e_n(x+m|q_n\al| )\eta_n +
  M\, o\left(q_n^{-n} \right)
&=
\widehat K  q_n N_{n}\int_{x}^{x+M/(q_n N_{n})}
e_n(t) dt +M\, o\left(q_n^{-n}\right)\\
&= 
\widehat K  q_n^{-n} N_{n} \int_{q_n x}^{q_n x+M/N_{n}}
\widetilde e_{n,\delta_n}(t) dt+M \, o\left( q_n^{-n}\right). 
\end{align*}  
}
 
The proof of \eqref{est_int-} is
similar.

To show (\ref{est_cond_b}), notice that under the assumption
$e_n(x+m\alpha)\geq 0 \ \text{ for all } m=k,\dots, k'$  
we have for these $m$ that  $\fp(x+m\alpha)\geq \bar \fp(x+m\alpha)$, hence
$$
\Sigma_x(k, k')=\sum_{m=k}^{k'} \ln\frac{\fq(x+m\alpha)}{\fp(x+m\alpha)}\leq 
\sum_{m=k}^{k'}\ln\frac{\bar \fq(x+m\alpha)}{\bar \fp(x+m\alpha)} =
\BS_x(M)\leq A.
$$
Estimate (\ref{est_cond_b-}) is proved in the same way.
\color{bleu1}  \end{proof} \black


\subsection{Proof of Proposition \ref{PropCond}.}\label{Sec_loc}

\begin{lemma}\label{lem_cond1} 
For 
$n$ sufficiently large, we have: 
$$
U_{1,n}\subset  \mathcal C_1(\fp, N_{1,n}).
$$
\end{lemma}
\begin{proof}
Fix $x\in I_{1,n}$ (the same argument holds for all $x\in U_{1,n}$).
Then $q_nx$ lies in an
interval of size 0.01 around the point $3/8$.  
Since $\widetilde{e}_{n, \delta_n}$ is smaller or equal to $\sin(8\pi x)$ for most of the above interval, we have that for large $n$  
$$
\int_{q_n x}^{q_n x+1/20}  \widetilde e_n(t) \leq 
\int_{3/8-0.01}^{3/8+0.04 }\sin 8 \pi t < -0.001.
$$
By (\ref{est_int}) with $M=N_{1,n}=N_{n}/20$, we have
\begin{align*}
\sum_{m=1}^{N_{1,n}} e_n(x+m\al ) K(x+m\a)  & = \widehat K q_n^{-n} N_{1,n} \int_{q_n x}^{q_n x+1/20}
\widetilde e_{n,\delta_n} (t) dt + { N_0 \,o(q_n^{-n})}\\   &< -0.001 \widehat K q_n^{-n} N_{1,n}  
\end{align*}
Since $|R(x,N_{1,n} )|\leq 4\kappa^{-2} N_{1,n} q_n^{-2n}$, and
$\kappa$ and $A$ are independent of
$N_n$, we get from (\ref{main_est}) for any $n$ sufficiently large:
$$
\Sigma_x( N_{1,n}) > 0.001 \widehat K  N_{1,n} q_n^{-n}.$$

{
Recall that, by \eqref{def_qn} and \eqref{def_Tn}, $N_n$ is of order
$q_n^{n^n}$ and $N_{1,n}=[N_n/20]$. Therefore,  $N_{1,n} \geq q_n^{n^6}/40\geq 
q_n^{6n}$, and for sufficiently large
$q_n$ we have $0.001 \widehat K  \sqrt{N_{1,n}} q_n^{-n} \geq 1$.
Hence, }
$$
\Sigma_x(N_{1,n}) > N_{1,n}^{1/2}.
$$
Likewise, $
\Sigma_x(-N_{1,n}) > N_{1,n}^{1/2}$.
\color{bleu1}  \end{proof} \black
\begin{lemma}\label{lem_cond2} For
$n$ sufficiently large, we have: 
$$
U_{2,n}\subset  \mathcal C_2(\fp, N_{2,n}, 1/n).
$$
\end{lemma}
\begin{proof} We choose $q_n$ and
$N_{2,n}$ satisfying (\ref{def_qn}) and (\ref{def_Tn}). 
Let $\BS_x(M)$ and $A$ be as in \S~\ref{sec.cob};
recall that $A$ 
only depends on $\bar \fp$. Assuming that $q_n$ is sufficiently large, 
we define
$$
L:=q_n^{n^2} >e^{e^A} .
$$
Since $N_{2,n}=q_n^{n^5}$ by (\ref{def_Tn}),  we have 
$$
N_{2,n}= L^{n^3} > L^{n^2},
$$
and $N_{2,n} \leq e^{L^{0.1}}$, as required in $\mathcal C_2(\fp, N_{2,n}, \frac1n)$.

Let $x\in I_{2,n}$ 
(the same argument holds for all $x\in U_{2,n}$). 
Then $$q_nx \in [5/16-0.01,5/16+0.01].$$ By the definition of $\widetilde
e_{n,\delta_n}$, for any 
$t\in [q_nx-0.001, q_nx+0.001]$ it holds that $\widetilde
e_{n,\delta_n}(t)  \geq 0.5$. Since by \eqref{def_Tn} we have 
$N_{2,n}/N_n<0.0001$, we get from  (\ref{est_int-}) with $M= L< N_{2,n}$
\begin{align*}
\sum_{m=-L+1}^{0} e_n(x+m\al ) K(x+m\al )  &= \widehat K q_n^{-n} N_{n} \int_{q_n x -L/N_n}^{q_n x}
\widetilde e_{n,\delta_n} (t) dt {+ N_0 \,o(q_n^{-n})}  \\
&\geq  \widehat K
0.49 L q_n^{-n}.
\end{align*}
Then, since $\widehat K>2$, we conclude from \eqref{main_est-} and \eqref{eq.R} that 
$$
\Sigma_x(-L) > 0.1 L q_n^{-n}  > \sqrt L.
$$
This gives $(\mathcal C_2a)$.

To verify $(\mathcal C_2b)$, notice that for  $x\in I_{2,n}$ and 
any  $m\in [-N_{2,n},N_{2,n}]$ we have \\$q_nx + m
q_n\al \in [5/16-0.02,5/16+0.02]$. Thus
$$
e_n(x+m\al)= q_n^{-n}\widetilde e_{n,\delta_n} (q_nx + m q_n\al) \geq 0.
$$
By (\ref{est_cond_b}), we have $(\mathcal C_2b)$, i.e., 
$$
\Sigma_x(k,k') \leq  A \ \text{ for all } \ -N_{2,n}\leq k\leq k'\leq N_{2,n}.
$$

To verify $(\mathcal C_2c)$, notice that for $L$ as above we have
$$
|L\alpha|= \frac{L}{q_n}|q_n\alpha| < \frac{L}{q_n} q_n^{-n^n} <
N_{2,n}^{-n^n/2}.
$$
Hence, for any $j\in [0,L-1]$, $k\in [-N_{2,n}, N_{2,n}]$ we have \\
$\DS
|\fp(x+j\a+kL\alpha)-\fp(x+j\a)|\leq  N_{2,n}^{-n^3}.
$
\color{bleu1}  \end{proof} \black

\begin{lemma}\label{lem_cond3} For any  $n$ sufficiently
  large,  we have: 
$$
U_{3,n}\subset  \mathcal C_3(\fp, {N_{3,n}},1/n).
$$
\end{lemma}

\begin{proof} Let $x\in I_{3,n}$ be fixed (the same argument holds for all $x\in U_{3,n}$). 
Define $Q=q_n$, and take for the numbers $u, v, w_\pm, u', v', w'_\pm$ ($v, v' \in [0.3,0.4]$) 
to be

\begin{align*}
u = \frac{1}{4} - xq_n, \quad  &v=\frac{3}{8}- xq_n, \quad
u'= \frac{1}{4} + xq_n, \quad  v' =\frac{3}{8}+ xq_n, \\
&w_\pm=v\pm\eps, \quad w'_\pm=v'\pm\eps,
\end{align*} 
where we set $\eps=\frac1n$.
Assume without loss of generality that for each of the numbers introduced above, its product
with $N_n$ is an integer that is a multiple of $q_n$.
Let $A>0$ be as in Proposition
\ref{prop_estimates}, and assume that $q_n$ is sufficiently large to
satisfy
$$
e^{e^A}< Q < N_n^{1/2}.
$$

The proof of  $(\C_3a)$ is  almost the same as the proof of $(\C_1)$ in
Lemma \ref{lem_cond1}. 
Namely, 
we have $x+vN_n\alpha = x+v/q_n +\cO(N_n^{-1}) = 3/(8q_n)+\cO(N_n^{-1})$.
Hence
\begin{align*}
\Sigma_x(vN_n, w_+N_n)&= \sum_{m=1}^{(w_+-v)N_n} 
\ln \fq\left(x+(vN_n+m) \alpha\right) -\ln \fp\left(x+(vN_n+m) \alpha\right)
\approx \\
&\sum_{m=1}^{\eps N_n} 
\ln \fq\left(3/(8q_n)+ m \alpha\right) -\ln \fp\left(3/(8q_n) +m \alpha\right)
=\Sigma_{3/(8q_n)}(\eps N_n).
\end{align*}
Notice that $3/(8q_n)\in I_{1,n}$, so the analysis of the latter sum
is analogous to that of Lemma \ref{lem_cond1}.  Let us repeat the argument.
The sum above is estimated using \eqref{est_int}.
Since  $1/n^2\ll \eps$, it follows from the definition of $\widetilde e_{n,\delta_n}$ that it is negative on most of the interval of integration $[ \frac38,  \frac38 +
\eps].$ {Moreover, on all the interval, if $\widetilde e_{n, \delta_n}(t)<0$ then 
  $\widetilde e_{n,\delta_n}(t)\leq   \sin 8 \pi t .$} 
  Thus
  \begin{align*}
\sum_{m=1}^{\eps N_n} &K(3/(8q_n)+m\al )e_n(3/(8q_n)+m\al ) = \widehat K q_n^{-n} N_n
\int_{3/8}^{3/8+\eps} 
\widetilde e_{n,\delta_n}(t) \, dt { + N_0 \,o(q_n^{-n})}\\
&<   \frac{\widehat K}{2} q_n^{-n} N_n
\int_{3/8+\eps/2}^{3/8+\eps} \sin(8\pi t) \, dt < 
- 0.001 \widehat K q_n^{-n} N_n \eps^2 .
\end{align*}
On the other hand $|R(x,\eps N_{n})|\leq 4\kappa^{-2} \eps N_{n} q_n^{-2n}$, and $A$ is independent of
$N_n$. Hence, by \eqref{main_est}, and since  $\widehat K \geq 2$  
$$
\Sigma_x(vN_n, w_+N_n) > 
 0.002 q_n^{-n} N_n \eps^2  -A - 4\kappa^{-2} \eps N_{n} q_n^{-2n} 
 > N_n^{1/2} 
$$
for $N_n$ sufficiently large. The remaining three estimates of this
item are proved in the same way.

To verify $(\mathcal C_3b)$, notice that, by 
the definition of $u'$ and $v$ we have: 
\begin{align*}
x - u' N_n\alpha&= x - u'/q_n+\cO(N_n^{-1})= -1/(4q_n)+\cO(N_n^{-1}), \\ x + vN_n\alpha&= x + v/q_n+\cO(N_n^{-1})= 3/(8q_n)+\cO(N_n^{-1}).
\end{align*}

Hence, for all $m\in [-u' N_n , v N_n]$  we have $e_n(x+m\alpha)\geq 0$. 
By (\ref{est_cond_b}), we have the first
part of $(\mathcal C_3b)$:
$$
\Sigma_x(k, k') \leq A  \ \text{ for all } \ -u' N_n \leq k\leq k'\leq v N_n.
$$
The second part of $(\mathcal C_3b)$ 
is verified in the same way using formula
(\ref{est_cond_b-}).

It remains to verify $(\mathcal C_3c)$. For $k \in [-v'N_n/Q,v'N_n/Q]$, 
take for $\BS(k)$ the sums $\BS_x(k)$ and let   $B_x(M):=\Sigma_x(M) - \BS_x(M)$. 
To verify (\ref{C3per}) notice that 
for each $l\in [-N_n/Q,N_n/Q]$ we have
$$\left| l Q\alpha \right| < \frac{1}{Q}.
$$ Therefore, since $\BS_x(M) =\ln g(x+\a)-\ln g(x+(M+1)\alpha)$
$$
|\BS_x(M)-\BS_x(M+lQ)| \leq 2\frac{\| \ln g \|_{C^1}}{Q} <  Q^{-1/2}
$$
if $Q$ is sufficiently large.

Next, we prove (\ref{C3B}). 
For each  $m\in [-u'N_n,uN_n]$  we have: $x+m\alpha \in
[-\frac{1}{4q_n},\frac{1}{4q_n}]$, and hence $e_n(x+m\alpha)=0$. 
This implies that $\fp(x+m\alpha)=\bar\fp(x+m\alpha)$, and
$$
\Sigma_x(M)=\BS_x(M) \ \ \text{ for all } \ \ M\in [-u'N_n,uN_n].
$$ 
For  $m\in [-u'N_n,vN_n]$  we have: $x+m\alpha \in
[-\frac{1}{4q_n},\frac{3}{8q_n}]$, and hence 
$e_n(x+m\alpha)\geq 0$. Then (\ref{est_cond_b}) implies, in particular, that
that for $M\in [0, vN_n]$ we have
$$
\Sigma_x(M) \leq \BS_x(M).
$$
For $m\in [-v'N_n,uN_n ]$ we have $e_n(x+m\alpha)\leq 0$, 
which implies the second part of (\ref{C3B}) by (\ref{est_cond_b-}).
This completes the proof of $(\mathcal C_3c)$.  
\color{bleu1}  \end{proof} \black

\bigskip \noindent {\bf \color{bleu1}  Proof of Proposition \ref{PropCond}.} Putting together Lemmas \ref{lem_cond1}, \ref{lem_cond2}, \ref{lem_cond3} immediately yields Proposition \ref{PropCond}. $\hfill \Box$


\section{Proofs of the main Theorems} 
\label{Sec_ABC}
\subsection{Proof of Theorem \ref{ThFlMV}.} \label{Sec_A}

By Theorem \ref{Cond_to_behaviors},  for any $\fp\in \cR$,  
for almost every $x\in \T$,  there are strictly increasing sequences 
of numbers $N_{j,n}$, such that for all $j=1,2,3$, $n\in \NN$ we have  
$$
x\in \mathcal C_j(\fp, N_{j,n}, 1/n).  
$$ 

For $j=1$, we have that $\sp_i=\fp(x+i\a)$ satisfies condition
 $\mathcal C_1(N_{1,n})$.  Hence, Proposition \ref{prop.localization} implies 
 Theorem \ref{ThFlMV}$(a)$ for $T=r_n:=e^{\sqrt{N_{1,n}/4}}$.

For $j=2$, we have that $\sp_i=\fp(x+i\a)$ satisfies condition
 $\mathcal C_2(N_{2,n},1/n)$.  
 
 Hence, Proposition \ref{prop.drift} implies Theorem \ref{ThFlMV}$(b)$ for $T=s_n:=N_{2,n}^5$ and $\eps_n=1/n$.
  
For $j=3$, function $\sp_i=\fp(x+i\a)$ satisfies 
$\mathcal C_3(N_{3,n},1/n)$. The conclusion of Proposition
\ref{prop.nolimit} holds for any
$T\in[N_{3,n}^5,e^{N_{3,n}^{1/4}}]$. Let us take $T=t_n:= N_{3,n}^5$. Then Proposition
\ref{prop.nolimit}       
 implies that for some $v=v(x) \in [0.3,0.4], v'=v'(x) \in [0.3,0.4]$ it holds 
\begin{equation*}
\begin{cases}
\Prob_x\left( Z_T \in [vN-\frac{N}{n},vN+\frac{N}{n}]   \right) > 0.1, \\ 
\Prob_x\left( Z_T \in [-v'N-\frac{N}{n},v'N+\frac{N}{n}]    \right) > 0.1.
\end{cases}
\end{equation*}
This proves Theorem \ref{ThFlMV}$(c)$ with
$b_n=vN_{3,n}\in[0.3T^{1/5},0.4 T^{1/5}]$, 
$b'_n=v'N_{3,n}\in[0.3T^{1/5},0.4 T^{1/5}]$,  
and $\eps_n=\frac1n$. 

Of course, choosing  larger values for
$T\in[N_{3,n}^5,e^{N_{3,n}^{1/4}}]$ allows to obtain a similar
statement to (c) with $b_n,b'_n$ of order $T^\delta$, for any
$0<\delta<1/5$.

$\hfill \Box$

{ \subsection{Proof of Theorems \ref{ThFlMV2} (a) and  \ref{th.asym} (a).}
 
We give the proof  Theorem \ref{ThFlMV2} (a). The proof of Theorem
\ref{th.asym} (a) is similar.  Fix $\a \notin \Q$. 
Let \begin{multline*} \cR_{n}=\bigg\{\fp \in \cP: \  \exists \sigma,
  \exists t>n,  \text{ such that } \forall x \in \T,
\forall z \in [-n,n], \\ 
\left| \Prob_x(Z_t<  \sigma \sqrt{t} z)-
\Phi(z)\right|<\frac{1}{n}
\bigg\}.
\end{multline*}
The set $\widetilde \cR = \cap_{n \geq 1}  \cR_{n}$ satisfies Theorem
\ref{ThFlMV2} (a). The sets $\cR_{n}$ are open, hence $\widetilde \cR$ is a
$G^\delta$ set.

It remains to show that $\widetilde \cR$ is dense in $\cP$.
By Proposition \ref{PrIM}, $\widetilde \cR$ contains all
  coboundaries. Recall that coboundaries are dense in $\cP$ by Lemma \ref{coboundaries}. Hence
$\widetilde \cR$ is a $G^\delta$-dense set. 
$\hfill \Box$

\subsection{Proof of Theorem \ref{ThFlMV2} (b).} 

Define 
$$
\begin{aligned}
\mathcal R_{v,\eps}= 
\bigg\{\fp \in \cP: \ &\exists
    {\rm \ open \ sets \ } \mathcal I \text{ and }\mathcal I' \text{
      with } {\rm Leb}({\mathcal
      I}>0.001, {\rm Leb}({\mathcal I'})>0.001 \\
&{\rm such \ that \ }
    \forall (x,x') \in {\mathcal I}\times {\mathcal I'}, \  \exists \
    \mu(x) > v^{1-\varepsilon}, \mu(x')=0, \\
& {\rm and \  for \ }   y \in \{x,x'\}, \ \forall \ z\in [-1/\eps,1/\eps] \ 
\eqref{EqOneSideDr2} {\rm \ holds \ with \ } T=v \bigg\}
\end{aligned}
$$

The sets $\mathcal R_{v,\eps}$ are open, 
and any $\DS \fp \in \hat{\cR} :=\bigcap_{\eps=\frac{1}{n}} \bigcup_v \mathcal
R_{v,\eps}$ satisfies Theorem \ref{ThFlMV2} (b). 
Hence, it suffices to
show that  $\fp(x)=\bar\fp(x)+e_n(x)$, as defined in  \eqref{def_pn},
belongs to $\mathcal R_{v_n,\frac 1 n}$, where we set
$v_n=N_{2,n}^5$. For this we take $\mathcal I_n =U_{2,n}$ and  apply
Proposition \ref{prop.drift} and get \eqref{EqOneSideDr2} with
$\mu_n(x)\geq v_n^{1-\eps_n}$  for every $x\in \mathcal I_n$ and every
$z\in \R$.

 On the other hand, we set $\mathcal I'_n= U_{3,n}$ and observe that  
for every $t\leq v_n$ we have $\fp(x+t\a)=\bar\fp(x+t\a)$. 
Hence the walk for such an $x$ up to time $v_n$ is the same as 
the one with the function $\bar \fp$ that is a coboundary. 
Since $N_{2,n}$ can be chosen arbitrarily large as function of $\bar
\fp$, we get  for every $x\in \mathcal I'_n$ \eqref{EqOneSideDr2} with $\mu_n(x)=0$. 

In conclusion, the set $\cR':=\bar \cR \cap \hat{\cR}$ satisfies the 
conditions of Theorem  \ref{ThFlMV2}. $\hfill \Box$

\subsection{Proof of Corollary \ref{CrNoIM}.}
\label{SSNoIM}
If the walk had an absolutely continuous stationary measure, then, 
by Proposition \ref{PrIM},
$\ln \fq-\ln \fp$ would be a smooth coboundary. Then, 
for any given sequence $\{u_N\}$ such that $u_N \to \infty$,  
we would have $\mu \{x: |\Sigma_x(N)|\geq u_N\} \to 0$ as $N \to \infty$.

However, Lemma \ref{lem_cond1} shows that if $\alpha$ is Liouville, then 
for a dense $G_\delta$ set of functions 
$\fp\in \cP$
there exists a sequence $\{N_j\}$ such 
that $|\Sigma_x(N_j)|>\sqrt{N_j}$ for a set of $x$ of measure
$0.01$. A contradiction. 
 $\hfill \Box$

\subsection{Proof of Theorem \ref{Th_B}.} 
Define
$$
\begin{aligned}
\cA_{m,n}=\bigg\{ \a \in \R : &\ \forall \fp \in \cP \  \exists \sigma  \text{ such that } \forall x \in \T,
\forall z \in [-n,n], \\ 
&\left| \Prob_x(Z_t<  \sigma \sqrt{t} z)-
\Phi(z)\right|<\frac{1}{n}\quad \text{for all } \ t\in [m,e^m]
\bigg \}.
\end{aligned}
$$
The set $\cA = \cap_{n \geq 1} \cup_{m\geq 1} \cA_{m,n}$ satisfies the
conclusion of the theorem. The sets $\cA_{m,n}$ are open hence $\cA$ is a
$G^\delta$ set.

By \cite{S2}, $\cA$ contains the Diophantine numbers. Hence
$\cA$ is a $G^\delta$-dense set. $\hfill \Box$

\subsection{Proof of Theorem \ref{th.asym} (b).}
\label{SSAsym}


Let
$\DS 
\lambda(x)=\frac{\fq(x)}{\fp(x)}.
$
To fix our notation, we assume that $\int \ln \lambda(x) dx=-c<0$ so
that the walk tends to $+\infty$.
The case $\int \ln \lambda(x) dx>0$ then follows by replacing $x$ by $-x$.
We want to perturb $\fp$ to get the behavior of Theorem \ref{th.asym} (b). 

\medskip

Let us first recall an important fact about the drift coefficient of an asymmetric walk in Theorem \ref{ThQLT}.
Following \cite{G-Int}, (see formula (1.6) and Theorem 4 of \cite{G-Int}), we associate to $\lambda$ 
a function 
\begin{equation}
\label{ueq}
 u(x)=1+2 \sum_{k=0}^\infty \prod_{j=0}^k \lambda(x-j\alpha).
 \end{equation}
Then the drift coefficient of the asymmetric walk in Theorem 
\ref{ThQLT} is given by the first integer $b_n(x)$ such that 
\begin{equation}
\label{QDrift}
\sum_{k=0}^{b_n(x)} u_\lambda(x+k\alpha) \geq n.
\end{equation}
The next lemma on the Birkhoff sums of a trigonometric polynomials will be a useful tool in our perturbation of $\fp$. 
 
\begin{lemma} \label{lemma.bound} Let $d,M>0$ and $q$ be such that $q>e^{e^{d+M}}$.
If $V$ is a trigonometric polynomial of degree $d$, and all the
coefficients of $V$ are bounded by $M$,
then for any $x \in \T$ 
$$
\left| \sum_{j=0}^{q-1} e^{V(x+j/q)} -q\int_{\T} e^{V(\theta)}d\theta
\right| <  e^{-q}.
$$
\end{lemma}

\begin{proof} First, expand 
$\DS e^{V(\cdot)}=\sum_{k=0}^N \frac{V^k}{k!}+\eps_N$,
where $N:=[2q/\ln q]$, so that  the error $\eps_N$
is small compared to $e^{-q}$. On the other hand, 
the polynomials $V^l$ that we keep are all of degree strictly less
than $q$, hence  
$\DS \sum_{j=0}^{q-1} V^l\left(x+\frac{j}{q}\right) =q\int V^l$. 
The lemma follows.  
\color{bleu1}  \end{proof} \black

Let us return to the proof of Theorem \ref{th.asym} (b).
As in the proof of Theorem \ref{ThFlMV2} (b), we only need to show 
density. Hence, by Lemma \ref{coboundaries}, we can start with a  
 $\bar \fp$ such that for some $c>0$
\begin{equation}
\label{LnLambdaCoB}
\ln  \bar \lambda(x) ={  \ln  \bar \fq(x) - \ln  \bar \fp(x) = -c+ \psi(x+\a)-\psi(x)}
\end{equation}
where $\psi$ a trigonometric polynomial. Let $b_n$ be as in 
Theorem \ref{th.asym} (a), see \eqref{eq.simpl2}. 
It is sufficient to prove that $\bar \fp$ can be perturbed into $\fp$ so that for an arbitrarily large ${t_n}$, and for some union of intervals  $\cJ_n$
and $\cJ'_n$ we have 
\begin{itemize}
\item[$(i)$] $\mu(\cJ_n)>0.8$ and $\mu(\cJ'_n)>0.1$;
\item[$(ii)$] For $x\in \cJ_n$ we have $|b_{{t_n}}(x)-b_n|<{t_n}^{1/4}$, and
  for $x\in \cJ'_n$ we have $b_{{t_n}}(x)> b_n+{t_n}^{0.9}$.
  \end{itemize}

We start by computing the drift corresponding to $\bar \fp$ for the special sequence of times: 
$$
{t_n}:=q_n^{n^2} \int_{\T} {\bbu(\theta)}d\theta, \quad n\in \N.
$$
\begin{lemma} \label{lemmabn} There is a constant $U$ (independent of $n$) such that for every $x \in \T$
\begin{equation*}
|\bar b_{t_n}(x) - q_n^{n^2}|\leq U.
\end{equation*}
\end{lemma}
\begin{proof} { Observe that if \eqref{LnLambdaCoB} holds then the function $\bar u$ associated to $\bar \lambda$ as in \eqref{ueq} can be written as 
\begin{equation}
\label{DriftCoB}
\bar u (x)= 1+ 2 \sum_{k=0}^{\infty} e^{-c(k+1)} e^{V_k(x)} 
\end{equation}
where $V_k(x)=\psi(x+\a)-\psi(x-k\a).$ 

Applying 
Lemma \ref{lemma.bound} to each term in \eqref{DriftCoB} (note that the norm of $V_k$ is bounded uniformly
in $k$), we conclude} that
 if $q_n$ is sufficiently large
then
$$
\left| \sum_{j=0}^{q_n-1} \bbu\left(x+\frac{j}{q_n}\right) -q_n\int_{\T}
{\bbu(\theta)}d\theta \right| <  e^{-q_n}.
$$

{ 
On the other hand, \eqref{def_qn} tells us that there is an integer
$p_n$ such that
$$
\left|\alpha-\frac{p_n}{q_n}\right|\leq q_n^{-n^4}. 
$$
Thus,
$$\left|\sum_{j=0}^{q_n-1} \bbu\left(x+j\a\right)-
\sum_{j=0}^{q_n-1} \bbu\left(x+\frac{p_n j}{q_n}\right)\right|\leq \|\bru\|_{C^1} q_n^{1-n^4}. $$
Observe that
$$ \sum_{j=0}^{q_n-1} \bbu\left(x+\frac{p_n j}{q_n}\right)=
 \sum_{j=0}^{q_n-1} \bbu\left(x+\frac{j}{q_n}\right),
$$
since as $j$ changes from $0$ to $q_{n-1}$ the set $p_n j$ goes over all possible residues mod $q_n.$ }
Therefore, for every $x$ in $\T$ we have for $n$ sufficiently large: 

\begin{equation*}
\left| \sum_{j=0}^{q_n-1} \bbu(x+j\a) -q_n\int_{\T} {\bbu(\theta)}d\theta \right|
<  q_n^{-n^3}.
\end{equation*}
{ Dividing an orbit of length $q_n^{n^2}$ into pieces of length $q_n$, we obtain}
\begin{equation*}
\left| \sum_{j=0}^{q_n^{n^2}} \bbu(x+j\a) -q_n^{n^2} \int_{\T}
{\bbu(\theta)}d\theta \right| <  q_n^{-n^3/2}.
\end{equation*}
This yields the conclusion of the lemma. \color{bleu1}  \end{proof} \black

Now we let $g_n$ be a smooth function satisfying
\begin{itemize}
\item[(a)] $\|g_n\|_{C^n} \leq 2^{-n}$;
\item[(b)] $g_n(\theta)=0$ for $\{q_n \theta\} \in[0,0.85]$;
\item[(c)] $g_n(\theta)=q_n^{-n-1}$ for $\{q_n \theta\} \in[0.86,0.99]$,
\end{itemize}
and let  $\fp(\theta)=\bar \fp(\theta)+g_n(\theta)$.
Define
$$\cJ_n=\{x \in \T : \{q_n x\} \in[0,0.84]\}, \quad
\cJ'_n=\{x \in \T : \{q_n x\} \in[0.86,0.98]\} . $$
We clearly have that $\mu(\cJ_n)>0.8$ and $\mu(\cJ'_n)>0.1,$ which is
$(i)$. To finish, we need to prove $(ii)$. 

{ Note that for $j \in [0,2q_n^{n^2}]$ we have 
$\DS q_n(x+j\a)=q_n x+jp_n+O\left(q_n^{2-n^4}\right).$}
Thus, it follows from (b) that for $x\in \cJ_n,$  $u(x+j\a) =
\bbu(x+j\a)$ for every $j \in [0,2q_n^{n^2}].$  Hence, we get  that 
$\DS b_{t_n}(x) = \bar b_{t_n}(x)$, and by Lemma \ref{lemmabn}, taking $b_n:=q_n^{n^2}$ we have  
\begin{equation}|b_{t_n}(x)-b_n| \leq U. \label{Ux} \end{equation}

On the other hand, for $x \in \cJ'_n$, we have from (c) that $u(x+j\a) \leq (1-q_n^{-n-2})
\bbu(x+j\a)$ for every $j \in [0,2q_n^{n^2}].$ { Hence,
$$ 
\sum_{j=0}^{b_n} u(x+j\a)\leq (1-q_n^{-n-2}) 
\left[\sum_{j=0}^{b_n}\bbu(x+j\a)\right]=t_n+O(1)-
\frac{t_n}{q_n^{-n-2}} <t_n-t_n^{0.95} .
$$
Therefore, for $x \in \cJ'_n$ we have:
$$
\sum_{j=b_n+1}^{b_n(x)} u(x+j\a)\geq t_n^{0.95},  
$$
and so $\DS b_n(x)>b_n+\frac{t_n^{0.95}}{\max_\theta u(\theta)}>b_n+{t_n}^{0.9}.$ 
Together with \eqref{Ux} this shows $(ii)$ and finishes the proof of  Theorem \ref{th.asym} (b)
$\hfill \Box$

\begin{appendix}
\section{\hskip-5mm Generic deterministic  environments.}

\label{SSGenApp}


Here we prove  Theorem \ref{ThGenErr}.
The main idea is the following. 
If we want to speed up the walk, we modify $\sp$
by adding a drift away from the origin, 
while to slow it down we increase the drift towards the origin.\footnote{\red The proof of  theorem \ref{ThGenErr} is a bit sketchy and the problem is that in the introduction we advise to start with this proof, and in the letter we say that nothing is anymore sketchy...}

\medskip 
\noindent {\bf Proof of $(a)$.} {\red We will use the notations and definitions of Section \ref{sec.mart}. }By \eqref{Rec}, the recurrence holds iff $M(n)\to\pm \infty$ as $n\to\pm\infty.$ 
The result follows since for each $R$ the condition that there is $n\in \NN$ such that
$M(n)>R$ and $M(-n)<-R$ is open and dense {\red for the product topology introduced in Definition \ref{Def_set_E}}. Openness is straightforward, and to obtain the density it is enough
to modify any given  $\sp$ to $\tilde \sp$ satisfying
\begin{equation}
\label{In23}
\tilde \sp(n)=\begin{cases} \frac{1}{3} & \text{for } n>K \\
\frac{2}{3} & \text{for } n<-K. \end{cases}
\end{equation}

\medskip 
\noindent {\bf Proof of $(b)$.} We also consider the environment given by \eqref{In23}. Note that for this environment
there are constants $C_1=C_1(K),$ and $C_2=C_2(K)$ such that 
$$ |M(n)|\geq C_1 e^{C_2|n|}.$$
Thus, for each $T$ and $r\geq 0$ 
$$
\Prob(|\barZ_T|\geq r)\leq \frac{1}{C_1 e^{C_2 r}}.
$$
It follows that for large $T,$ \eqref{EqGenLoc} is satisfied, showing the density of this condition. 
The openness is also clear.

\medskip 
\noindent {\bf Proof of $(c)$.} It is sufficient to show that for each $\eps$ the set of environments
such that for some $T$
$$ \sup_z \left|\Prob\left(\frac{\barZ_T-\frac{T}{3}}{\sqrt{\frac{8T}{9}}}\leq z\right)-\Phi(z)\right|<\eps $$
is dense.
We now modify any given environment so that $\tilde \sp(n)=\frac{2}{3}$ for $|n|>K.$ 
Then the walk spends a finite time to the left of $K.$  It follows that 
$$\Prob\left(\frac{\barZ_T-\frac{T}{3}}{\sqrt{\frac{8T}{9}}}\leq z\right)\to \Phi(z)$$
uniformly in $z$ as needed.

To prove part $(d)$, we modify a given environment outside $[-K, K]$ in three steps. 
First we take $K_1\gg K$ and modify $\sp$ on $[-K_1, K_1]\setminus [-K, K]$ to achieve that
$$
\sum_{j=1}^{n} \ln \tilde\fq(j)-\ln\tilde\fp(j)=\sum_{j=-n+1}^{0} \ln
\tilde\fp(j) -\ln \tilde\fq(j) ,
$$
where $\tilde\sq(j)=1-\tilde\sp(j).$
Next we take $K_2\gg K_1$ and let $\tilde\sp(n)=\frac{1}{2}$ if $|n|\in [K_1+1, K_2].$
Finally, we let $\tilde\sp(n)=\frac{1}{3}$ if $n<-K_2$ and $\tilde\sp(n)=\frac{2}{3}$ if $n>K_2.$
It is easy to see that, given $\eps>0$, we can make $K_1$ and $K_2$ so large that 
\begin{equation}
\label{ScalesEqual}
1-\eps< \frac{|M_-|}{|M_+|}
<1+\eps,
\end{equation}
where $M_+$ and $M_-$ are defined in \eqref{EscRight}. Then \eqref{EscRight} shows that
$$ \Prob\left(\lim_{t\to\infty} \barZ_T=+\infty\right)=\frac{|M_+|}{|M_+|+|M_-|} 
\in \left[\frac{1}{2+\eps}, \frac{1}{2-\eps}\right].
$$
The same holds for $ \DS \Prob\left(\lim_{t\to\infty} \barZ_T=-\infty\right)$.

On the other hand, it is easy to see that
$$ \Prob\left(\frac{\barZ_T-\frac{T}{3}}{\sqrt\frac{8T}{9}}\leq z\left|\lim_{t\to\infty} \barZ_t=+\infty\right.\right)=\Phi(z)$$ 
and 
$$\Prob\left(\frac{\barZ_T+\frac{T}{3}}{\sqrt\frac{8T}{9}}\leq z\left|\lim_{t\to\infty} \barZ_t=-\infty\right.\right)=\Phi(z) . $$
It follows that for $T$ sufficiently large \eqref{EqGenTwoSideDr} is satisfied with $b(T)=\frac{T}{3},$
$\eps(T)=T^{-1/3}$ proving the density of this condition. \hfill $\Box$ 

\end{appendix}

\bigskip 

\noindent{\sc \color{bleu1} Acknowledgement.} We are grateful for two anonymous referees who made numerous comments and suggestions that helped us make substantial revisions to the first version of this paper.

\vspace{.6cm}

\noindent
Dmitry Dolgopyat\\ 
University of Maryland \\
Maryland, United States\\
email:\,{\it {dolgop@umd.edu}}

\vspace{.6cm}

\noindent
Bassam Fayad\\ 
CNRS UMR7586 -- IMJ-PRG \\
Paris, France\\
email:\,{\it{bassam.fayad@imj-prg.fr}}

\vspace{.6cm}

\noindent
Maria Saprykina \\ 
Institute f\"{o}r Matematik, KTH  \\
Stockholm, Sweden\\
email:\,{\it {masha@kth.se}}

\end{document}